\documentclass[letterpaper,11pt]{amsart}

\usepackage{amsmath, amsxtra, amsthm, amssymb,mathtools,bm,amsfonts}
\usepackage{graphicx}
\usepackage{url}
\usepackage{tikz-cd}
\usepackage{tikz}
\usepackage[T1]{fontenc}

\usepackage{enumitem}
\usepackage[pdftex,hidelinks,backref=page]{hyperref}
\hypersetup{
    colorlinks,
    citecolor=magenta,
    filecolor=magenta,
    linkcolor=blue,
    urlcolor=black
}

\usepackage{mathtools}
\usepackage{esvect}
\usepackage{multirow}
\usepackage{listings}
\usepackage{accents}
\usepackage{subcaption}
\usepackage[cmtip,all]{xy}

\DeclareMathOperator*{\CC}{\mathbb{C}}
\DeclareMathOperator*{\ZZ}{\mathbb{Z}}
\DeclareMathOperator*{\codim}{\mathrm{codim}}
\DeclareMathOperator*{\Sym}{\mathrm{Sym}}
\DeclareMathOperator*{\Wr}{\mathrm{Wr}}
\DeclareMathOperator*{\Gr}{\mathrm{Gr}}
\DeclareMathOperator*{\GL}{\mathrm{GL}}

\DeclareMathOperator*{\PGL}{\mathrm{PGL}}
\DeclareMathOperator*{\Tr}{\mathrm{Tr}}
\DeclareMathOperator*{\PhiNew}{\Phi_{\mathrm{N}}}
\DeclareMathOperator*{\FNew}{F_{\mathrm{N}}}
\DeclareMathOperator*{\PhiWall}{\Phi_{\mathrm{W}}}
\DeclareMathOperator*{\FWall}{F_{\mathrm{W}}}

\theoremstyle{definition}
  \newtheorem{mydef}{Definition}
  \numberwithin{mydef}{section}

\theoremstyle{plain}
  \newtheorem{mythm}[mydef]{Theorem}
  \newtheorem{myprop}[mydef]{Proposition}
  \newtheorem{mylemma}[mydef]{Lemma}
  \newtheorem{mycor}[mydef]{Corollary}

\theoremstyle{remark}
  \newtheorem{myrmk}[mydef]{Remark}
  \newtheorem{myexample}[mydef]{Example}

\numberwithin{equation}{section}

\title{On the PGL\textsubscript{2}-equivariant intersection theory of G\lowercase{r}(2,4)}

\author{Yuxuan Sun}
\address{Department of Mathematics, University of Toronto, Toronto, ON M5S 2E4, Canada}
\email{\href{mailto:austin.sun@mail.utoronto.ca}{austin.sun@mail.utoronto.ca}}
\thanks{The author was partially supported by a NSERC Canada Graduate Scholarship - Doctoral (CGS D)}

\subjclass[2020]{Primary 14C15, 14C17. Secondary 14D23, 14L24.}

\begin{document}

\begin{abstract}
We determine the $\mathrm{PGL}_2$-equivariant Chow ring of $\mathrm{Gr}(2,4)^s$, the $\mathrm{PGL}_2$-stable locus of $\mathrm{Gr}(2,4)$, over any algebraically closed based field of characteristic not equal to 2 or 3. In the process, we demonstrate that the quotient stack $[\mathrm{Gr}(2,4)^s/\mathrm{PGL}_2]$ can be presented as the quotient of an open subset of $\mathbb{P}^1$ by a suitably chosen $S_4\leq \mathrm{PGL}_2$. We also discuss some apparent difficulties with computing the full $\mathrm{PGL}_2$-equivariant Chow ring of $\mathrm{Gr}(2,4)$.
\end{abstract}

\maketitle
\setcounter{tocdepth}{1}
\tableofcontents

\section{Introduction}
\label{section-intro}

Let $k$ be an algebraically closed field of characteristic not equal to 2 or 3, and fix $W=k[t_0,t_1]_1$ to be the two-dimensional vector space over $k$ spanned by the indeterminates $t_0$ and $t_1$. The goal of this paper is to give a presentation for the $\PGL_2$-equivariant Chow ring of the $\PGL_2$-stable locus of $\Gr(2, \Sym^3 W)$, which is an open subset of $\Gr(2,4)$ consisting of all but six $\PGL_2$-orbits according to results from geometric invariant theory \cite{Ne81}. In particular, we shall compute this $\PGL_2$-equivariant Chow ring by establishing the isomorphism of quotient stacks
\begin{align}
[\Gr(2,4)^s/{\PGL}_2] &\cong [(\mathbb{P}^1 - \FWall) / S_4], \label{key-quot-stack-isom}
\end{align}
where $\Gr(2,4)^s$ denotes the $\PGL_2$-stable locus, $\FWall \subset \mathbb{P}^1$ is a fixed finite subset defined in Remark \ref{rel-bw-Newstead-and-Wall-reps}, and $S_4 \leq \PGL_2$ acts on $\mathbb{P}^1$ via the induced action of its unique two-dimensional irreducible representation. 

\subsection{Background}
We begin by identifying $\mathbb{P}^n:=\mathbb{P}(\Sym^n W)$ with the moduli space of degree $n$ divisors on $\mathbb{P}^1:= \mathbb{P}W$ via
\[
\mathbb{P}({\Sym}^n W) \cong {\Sym}^n \mathbb{P}W.
\] 
One may obtain an explicit $\PGL_2$-stratification of this moduli space using $\PGL_2$-invariant subvarieties known as coincident root loci: as introduced in \cite{FNR06}, these subvarieties are defined by
\[
X_{\lambda} := \left\{ B(t_0, t_1) \in k[t_0,t_1]_n \mid B = \prod_{j=1}^d L_j^{\lambda_j} \text{ for some linear forms $L_j$} \right\}
\]
for every partition $\lambda = (\lambda_1, \dots, \lambda_d)$ of $n$. In other words, $X_{\lambda}$ is the collection of all points in $\mathbb{P}(\Sym^n W)$ corresponding to degree $n$ divisors on $\mathbb{P}^1$ of type $\lambda=(\lambda_1, \dots,\lambda_d)$. In \cite{FNR06}, Feh\'{e}r, N\'{e}methi and Rim\'{a}nyi gave explicit formulae for the $\GL_2$-equivariant Poincar\'{e} duals of the $X_{\lambda}$'s, and it was also shown that these $\GL_2$-equivariant Poincar\'{e} duals generate the $\GL_2$-equivariant cohomology and Chow rings of $\mathbb{P}(\Sym^n W)$. Explicit presentations for the $\PGL_2$-equivariant Chow ring of $\mathbb{P}^n$ were also given by Spink and Tseng in \cite{ST22} using the formulae for the Poincar\'{e} duals of the $X_{\lambda}$'s found by \cite{FNR06}. 

By extension, one may study the $\PGL_2$-invariant subvarieties of $\Gr( d, {\Sym}^n W)$ via the Wronskian map, $\Wr: \Gr( d, {\Sym}^n W) \rightarrow \mathbb{P}({\Sym}^N W)$ given by
\[
\Wr(\langle f_1, \dots, f_d \rangle) := \mathrm{Jac}(f_1, \dots,f_d)
\]
where $\mathrm{Jac}(f_1, \dots,f_d)$ is the Jacobian of the polynomials $f_1,\dots, f_d$, and $N=d(n+1-d)$ is the dimension of the Grassmannian. As a morphism of schemes, the Wronskian map is $\PGL_2$-equivariant, flat and finite of degree equal to the number of rectangular $d\times(n+1-d)$ standard Young tableaux (Lemma 1.1.3 of \cite{Me04}, cf. \cite{EH83}). Using the Wronskian and GIT, Newstead \cite{Ne81} and Wall \cite{Wa98} were able to classify the $\PGL_2$-orbits of $\Gr(2,4)$ and $\Gr(2,5)$ respectively. In addition, Meulien \cite{Me04} calculated the ring of $\PGL_2$-invariants for $\Gr(2,6)$.

Hence, it is natural to ask for ring presentations of $A^*_{\PGL_2}(\Gr( d, {\Sym}^n W))$, the $\PGL_2$-equivariant Chow ring of $\Gr( d, {\Sym}^n W)$. In general, one cannot use the pullback of the Wronskian to deduce a ring presentation for $A^*_{\PGL_2}(\Gr( d, {\Sym}^n W))$ directly from that of $A^*_{\PGL_2}(\mathbb{P}^{d(n+1-d)})$, given that the degree of $\Wr$ is strictly greater than 1 for each $d\geq 2$ and $n\geq 3$. Therefore, more work needs to be done for analyzing the $\PGL_2$-invariant subvarieties of $\Gr( d, {\Sym}^n W)$.
\subsection{Challenges with calculating $A^*_{\mathrm{PGL}_2}(\mathrm{Gr}(2, 4))$}

Whenever $n$ is even, one may directly determine the ring structure of $A^*_{\PGL_2}(\Gr( d, {\Sym}^n W))$ using $\PGL_2$-representation theory: we may view $\Gr( d, {\Sym}^n W)$ as a $\PGL_2$-equivariant Grassmann bundle of 2-planes in ${\Sym}^n W$ over a point, and we carry out this computation for each $n$ even in Proposition \ref{integral-Chow-ring-of-Grassmannian-equivar-PGL2-bd}. However, such computations become highly non-trivial for $n$ odd, given that ${\Sym}^n W$ is no longer a $\PGL_2$-representation. Even for the first non-trivial case $\Gr(2, {\Sym}^3 W) = \Gr(2,4)$, computations for its $\PGL_2$-equivariant Chow ring do not seem to exist in the current literature. As such, we shall focus on studying the $\PGL_2$-equivariant intersection theory of $\Gr(2,4)$. 

In general, $\PGL_2$-equivariant Chow rings are difficult to compute since $\PGL_2$ is not a special group: there exists $\PGL_2$-torsors over certain schemes that are  \'{e}tale-locally trivial but not Zariski-locally trivial. As such, when computing the $\PGL_2$-equivariant Chow ring of a smooth scheme $X$, it is useful to regard $A^*_{\PGL_2}(X)$ as the Chow ring of the quotient stack $[X/\PGL_2]$ and attempt to rewrite the stack presentation as a quotient by some special subgroup (i.e., free of $\PGL_2$). For example, in \cite{DL21}, Di Lorenzo computed the integral Chow ring of the stack $\mathcal{H}_g$ of hyperelliptic curves of odd genus by discovering a way of writing $\mathcal{H}_g$ as a $GL_3\times \mathbb{G}_m$-quotient stack as opposed to the previously known $\PGL_2\times \mathbb{G}_m$-quotient presentation. In \cite{AI19}, Asgarli and Inchiostro computed the Picard group of the classifying stack of complete intersections of pairs of quadrics in $\mathbb{P}^n$ by discovering a different quotient stack presentation of this stack other than the natural quotient by $\PGL_n$. Other similar examples of computing Chow rings of quotient stacks involving $\PGL_2$-presentations include \cite{CL25}. 

Although we are not able to find a ring presentation for $A^*_{\PGL_2}(\Gr(2,4))$, we are able to determine the $\PGL_2$-equivariant Chow ring of the stable locus $\Gr(2,4)^s$ by rewriting the quotient stack $[\Gr(2,4)^s/\PGL_2]$ in terms of a quotient by a fixed finite subgroup of $\PGL_2$ as in Equation \ref{key-quot-stack-isom}. There does not seem to be an analogous isomorphism of quotient stacks for the semistable locus $\Gr(2,4)^{ss}$ or for the entire Grassmannian $\Gr(2,4)$: see Section \ref{statement-of-main-results} for detailed discussions. Hence, for the rest of this paper, we shall focus on the stable locus $\Gr(2,4)^s$. Given that $\Gr(2,4)^s$ consists of all but six $\PGL_2$-orbits \cite{Ne81}, computing $A^*_{\PGL_2}(\Gr(2,4)^s)$ would give us a wealth of information about the $\PGL_2$-equivariant intersection theory of $\Gr(2,4)$ already. In this way, we believe that the problem of computing $A^*_{\PGL_2}(\Gr(2,4)^s)$ is interesting in its own right.

We conclude this section by discussing the apparent difficulties with computing $A^*_{\PGL_2}(\Gr(2,4))$. One feasible approach for computing $A^*_{\PGL_2}(\Gr(2,4))$ seems to be the stratification method. In this case, the natural $\PGL_2$-equivariant strata would consist of the open stratum $\Gr(2,4)^s$ plus the six distinct orbit closures $\overline{Z_1}, \overline{Z^{(0)}_2}, \overline{Z^{(1)}_2}, \overline{Z^{(2)}_2}, \overline{Z^{(1)}_3}, \overline{Z^{(2)}_3}$, where each $Z^{(j)}_i$ is a distinct non-$\PGL_2$-stable orbit of dimension $i$ (Section 4 of \cite{Ne81}). However, implementing the stratification method is challenging since the $\overline{Z^{(j)}_i}$'s do not form a strictly increasing chain by containment and have non-empty intersections: as shown in Section \ref{difficulties-entire-Chow-ring}, the inclusion relations between the six closed strata are given by Figure \ref{fig-PGL2-non-stable}, which demonstrates that all 2- and 3-dimensional closed strata have non-empty intersections. In this way, it is difficult to find relations between pushforwards of generators coming from the Chow rings of different strata since the relevant excision and Mayer-Vietoris exact sequences do not give any information about these relations. More details are given in Section \ref{difficulties-entire-Chow-ring}.

\subsection{Organization}
The rest of this paper is organized as follows. In Section \ref{section-basics}, we give the necessary background for equivariant intersection theory, followed by a summary of useful results from the classification of $\PGL_2$-orbits of $\Gr(2,4)$ by Newstead \cite{Ne81}. Then, in Section \ref{statement-of-main-results}, we state our main results regarding Equation \ref{key-quot-stack-isom} and our ring presentation for $A^*_{\PGL_2}(\Gr(2,4)^s)$. Section \ref{description-S4-symmetry-stable-locus} gives a description of the $S_4$-symmetry on $\Gr(2,4)$. Then, Section \ref{description-of-equivar-Chow-ring-of-locus-of-stable-pts-S4-symm} gives the proof of Theorem \ref{main-thm-S4-symmetry-stable-locus}, and Section \ref{computation-equivar-Chow-ring-of-locus-of-stable-pts} gives the proof Theorem \ref{main-thm-Chow-ring-stable-locus}. Finally, Section \ref{section-future-directions} discusses interpretations of the generators of the ring $A^*_{\PGL_2}(\Gr(2,4)^s)$ that we find, followed by a discussion of difficulties of determining the entire $\PGL_2$-equivariant Chow ring $A^*_{\PGL_2}(\Gr(2,4))$.

\subsection*{Acknowledgements}
The author would like to thank his PhD thesis supervisor, Prof. Hunter Spink, for suggesting this project and for the numerous helpful discussions leading to the completion of this paper. In addition, the author would like to acknowledge partial funding support by the NSERC in the form of a NSERC Canada Graduate Scholarship - Doctoral (CGS D) for this project.

\section{Basic definitions and results}
\label{section-basics}

\subsection{Results from equivariant intersection theory}
\label{subsection-basics-EIT}

In this section, we summarize some basic results from equivariant intersection theory needed for later sections. Our main sources of reference for equivariant intersection theory are the paper by Edidin and Graham \cite{EG98} as well as the book by Totaro \cite{To14}, and we also refer to the book on equivariant cohomology by Anderson and Fulton \cite{AF24} for analogous results in the case of cohomology. 

For simplicity, we will assume that all schemes considered are smooth, and that actions of group schemes are smooth actions: this is the case for $\PGL_2$-actions on Grassmannians $\Gr(d, \Sym^n W)$. In particular, there exists a well-defined graded ring structure on every $G$-equivariant Chow ring $A^*_{G}(X)$ whenever $G$ acts smoothly on a smooth scheme $X$ \cite{EG98}. Moreover, $G$-equivariant Chern classes of any $G$-equivariant vector bundle $E$ over $X$ are elements of $A^*_{G}(X)$ \cite{EG98}: we shall denote these classes by
\[
c_i^G(E) \in A^i_{G}(X)
\]
for $i=1, \dots, \mathrm{rank}(E)$.

Results from \cite{EG98} allows one to define the Chow ring of the quotient stack $[X/G]$ as
\[
A^*([X/G]) \cong A^*_{G}(X)
\]
whenever $[X/G]$ is smooth, and it is also shown in \cite{EG98} that the above Chow ring does not depend on presentations of the quotient stack. Hence, for our purposes, we may write
\[
A^*_{\PGL_2}(\Gr(d, {\Sym}^n W)) \cong A^*([\Gr(d, {\Sym}^n W)/{\PGL}_2])
\]
and analogously for $A^*_{\PGL_2}(\Gr(2,4)^s)$ when finding its presentation via the isomorphism given by Equation \ref{key-quot-stack-isom}.

Denote the $G$-equivariant Chow ring of a point by
\[
A^*_G := A^*_G(\{ \text{pt} \}),
\]
and notice that it has a well-defined graded ring structure for every smooth group scheme $G$. We may also relate equivariant Chow rings of different smooth schemes by pullback and pushforward maps. For any $G$-equivariant morphism $f:X \rightarrow Y$ between smooth schemes $X$ and $Y$, there exists a well-defined pullback ring homomorphism $f^*: A^*_G(Y) \rightarrow A^*_G(X)$ \cite{EG98}. In addition, for any proper $G$-equivariant morphism $f: X\rightarrow Y$ between $X$ and $Y$ smooth, there exists a pushforward map on equivariant Chow groups $f_*: A^i_G(X) \rightarrow A^{i+d}_G(Y)$ with $d=\dim Y - \dim X$ given by the relevant Gysin pushforward maps \cite{EG98}. In this case, $f_*$ is a $A^*_{G}$-module homomorphism for each $i$ but not necessarily a ring homomorphism. Finally, whenever $H\leq G$ is a closed subgroup and $X$ is a smooth $G$-scheme, we may construct a pullback $\iota^*: A^*_{G}(X) \rightarrow A^*_{H}(X)$ from the inclusion map of groups $\iota: H\xhookrightarrow{} G$. 

For our proof of Theorem \ref{main-thm-Chow-ring-stable-locus}, we shall need the following result of \cite{Ve98} calculating the $\PGL_2$-equivariant Chow ring of a point with integral coefficients. This result was originally obtained by Pandharipande in \cite{Pa98} using the isomorphism $\PGL_2(\CC) \cong SO(3)_{\CC}$.

\begin{mythm}[\cite{Ve98}, Theorem 1]
\label{integral-Chow-ring-of-classifying-space-of-PGL2}

Suppose that the base field $k$ is of characteristic not equal to 2. Denote the adjoint representation of $\PGL_2(k)$ by $\mathfrak{sl}_2$ (notice that $\mathfrak{sl}_2 \cong \Sym^2 W$ as $PGL_2$-representations). Then, the $\PGL_2$-equivariant Chow ring of a point is given by
\[
A^*_{\PGL_2} \cong \ZZ[c_2(\mathfrak{sl}_2),c_3(\mathfrak{sl}_2)]/(2c_3(\mathfrak{sl}_2)).
\]
\end{mythm}

We remark that the original result of \cite{Ve98} is stated over $\CC$, but the ring presentation is the same whenever the base field is algebraically closed and of characteristic not equal to 2. 

To conclude this section, we give a presentation of $A^*_{\PGL_2}(\Gr(2,\Sym^n W))$ for each $n$ even by considering $\Gr(2,\Sym^n W)$ as a $\PGL_2$-equivariant Grassmann bundle.

\begin{myprop}
\label{integral-Chow-ring-of-Grassmannian-equivar-PGL2-bd}

Denote $\Gr(2,n+1) = \Gr(2, \Sym^n W)$. Then, for each $n$ even:
\[
A^*_{\PGL_2}(\Gr(2,n+1)) 
\cong A^*_{\PGL_2}[s_1,s_2,q_1,\dots,q_{n-1}]/(s\cdot q = c^{\PGL_2}({\Sym}^n W))
\]
where
\[
s := c^{\PGL_2}(\mathcal{S}), q := c^{\PGL_2}(\mathcal{Q})
\]
are the $\PGL_2$-equivariant total Chern classes of the (rank 2) universal subbundle $\mathcal{S}$ and the (rank $n-1$) quotient bundle $\mathcal{Q}$ respectively.
\end{myprop}
\begin{proof}
Since $\Sym^n W$ is a $\PGL_2$-representation if and only if $n$ is even, $\Gr(2,n+1)$ is a $\PGL_2$-equivariant Grassmann bundle over a point if and only if $n$ is even. By viewing $\Gr(2,n+1)$ as a $\PGL_2$-equivariant Grassmann bundle over a point whenever $n$ is even, we obtain the above presentation for $A^*_{\PGL_2}(\Gr(2,n+1))$ using the Grassmann bundle formula for Chow rings found in Theorem 9.18 of \cite{EH16} along with the $\PGL_2$-representation theory of $\Sym^n W$. We remark that the universal subbundle $\mathcal{S}$ and the universal quotient bundle $\mathcal{Q}$ over $\Gr(2,n+1)$ are $\PGL_2$-equivariant if and only if $n$ is even, since $\PGL_2$-equivariance of $\mathcal{S}$ and $\mathcal{Q}$ follows from $\Sym^n W$ being a $\PGL_2$-representation.
\end{proof}

Therefore, when the characteristic of the base field is not equal to 2, we obtain a ring presentation for $A^*_{\PGL_2}(\Gr(2,n+1))$ by combining Theorem \ref{integral-Chow-ring-of-classifying-space-of-PGL2} with Proposition \ref{integral-Chow-ring-of-Grassmannian-equivar-PGL2-bd}.

\subsection{Results from geometric invariant theory for pencils of binary cubics}
\label{subsection-basics-GIT-for-pencils-of-binary-cubics}

In this section, we give a brief summary of useful results regarding the classification of $\PGL_2$-orbits of $\Gr(2,4)$ from the paper by Newstead \cite{Ne81} and the paper by Wall \cite{Wa83}. Their results are mainly based on geometric invariant theory (GIT), for which we use the book by Newstead \cite{Ne78} as our main source of reference. 

For the following, let $k$ be an algebraically closed field of characteristic not equal to 2 or 3. We recall the $\PGL_2$-invariants of binary quartics in $\mathbb{P}(\Sym^4 W)$: these invariants were used by Newstead to construct $\PGL_2$-invariants of pencils of binary cubics \cite{Ne81}. A binary quartic of the form (considered as a point in $\mathbb{P}(\Sym^4 W)$)
\[
c_0t_0^4+4c_1t_0^3t_1+6c_2t_0^2t_1^2+4c_3t_0t_1^3+c_4t_1^4
\]
has $\PGL_2$-invariants given by
\begin{align*}
I &:= c_0c_4-4c_1c_3+3c_2^2, \\
J &:= c_0c_2c_4+2c_1c_2c_3-c_0c_3^2-c_1^2c_4-c_2^3,
\end{align*}
and it is shown in \cite{Ne78} that $I, J$ generate the ring of $\PGL_2$-invariants of binary quartics. 

Next, let $p\in\Gr(2,4)$ be the pencil generated by distinct binary cubics
\[
f=a_0 t_0^3 + a_1 t_0^2t_1 + a_2 t_0t_1^2 + a_3t_1^3, 
g=b_0 t_0^3 + b_1 t_0^2t_1 + b_2 t_0t_1^2 + b_3t_1^3.
\]
From Newstead \cite{Ne81}, the Wronskian map $\Wr: \Gr(2,4) \rightarrow \mathbb{P}^4$ may be written explicitly as
\[
\Wr: \Gr(2,4) \rightarrow \mathbb{P}^4, p \mapsto (p_{01}:2p_{02}:3p_{03}+p_{12}:-2p_{31}:p_{23}),
\]
with $p_{ij}:= a_ib_j-a_jb_i$ being the Pl\"{u}cker coordinates of $p$ defined by the coefficients of $f,g$.

Then, the $\PGL_2$-invariants of $p$ are given by
\[
I' := 3p_{03}-p_{12}, \text{ where } I'^2 = 12I,
\]
and $J$ calculated as above using the Pl\"{u}cker coordinates of $p$ \cite{Ne81}. It is shown in \cite{Ne81} that $I',J$ generate the ring of $\PGL_2$-invariants of pencils of binary cubics, and classification of all the $\PGL_2$-stable orbits may also be obtained from these $\PGL_2$-invariants.

\begin{mythm}[\cite{Ne81}]
\label{GIT-quotient-pencils-of-binary-cubics}

Let
\[
V = \{ p \in \Gr(2,4) \mid I', J \text{ are not both zero at } p\} \subset \Gr(2,4),
\]
and note that this is a Zariski-open subset of $\Gr(2,4)$. Also, define a morphism
\[
\PhiNew: V \rightarrow \mathbb{P}^1, 
p \mapsto (I'^3|_p:J|_p).
 \]
Then, $(\mathbb{P}^1, \PhiNew)$ is a good quotient of $V$ by $\PGL_2$.

In particular: 
\begin{enumerate}
\item The locus of $\PGL_2$-semistable points of $\Gr(2,4)$ is given by
\[
\Gr(2,4)^{ss} = V.
\]
\item Let $\FNew:= \{ (\pm 6^3: 1) \} \subset \mathbb{P}^1$. Then, the locus of $\PGL_2$-stable points of $\Gr(2,4)$ is given by
\[
\Gr(2,4)^s = \{ p \in \Gr(2,4) \mid (I'|_p, J|_p) \neq (0,0), I'^3|_p \neq \pm 6^3 J|_p \},
\]
and $\PhiNew$ restricted to $\Gr(2,4)^s$ is a geometric quotient with image
\[
\PhiNew(\Gr(2,4)^s) = \mathbb{P}^1 - \FNew.
\]
\end{enumerate}
\end{mythm}

The above theorem allows one to compute unique representatives for each one of the stable $\PGL_2$-orbits. Wall (\cite{Wa83}, cf. \cite{Wa98}) also obtained a one-parameter family of representatives for the stable $\PGL_2$-orbits with particularly simple expressions. For future reference, we record Wall's representatives and give a proof sketch that they indeed exist in every $\PGL_2$-stable orbit.

\begin{myprop}[\cite{Wa83}, cf. \cite{Wa98}]
\label{Wall-reps-stable-orbits}

For any $\rho\in k$, denote
\[
p_{\rho} = \langle f_{\rho}, g_{\rho} \rangle := \langle t_0^3+\rho t_0t_1^2, \rho t_0^2t_1 + t_1^3 \rangle \in \Gr(2,4),
\]
and set
\[
p_{\infty} :=  \langle t_0t_1^2, t_0^2t_1 \rangle \in\Gr(2,4).
\]
Then, for every $\PGL_2$-stable point $p\in \Gr(2,4)$, there exists some $p_{\rho}$ of the above form such that
\[
p \in O_{\PGL_2}(p_{\rho}),
\]
where $p_{\rho}$ is such that
\[
\rho \neq 0, \infty, \pm 1, \pm 3.
\]
\end{myprop}
\begin{proof}
The $\PGL_2$-invariants $I',J$ evaluated at $p_{\rho}$ are given by
\begin{align}
I'|_{p_\rho} &= 3+\rho^2, \label{I-prime-for-Wall-rep}\\
J|_{p_\rho} &= \frac{1}{6^3}(\rho^2-3)(\rho^2-6\rho-3)(\rho^2+6\rho-3). \label{J-prime-for-Wall-rep}
\end{align}
Hence, from direct computations using the good quotient $(V,\PhiNew)$ from Theorem \ref{GIT-quotient-pencils-of-binary-cubics}, one may verify that
\begin{align*}
\PhiNew(\{ p_{\rho} \mid \rho\neq 0, \infty, \pm 1, \pm 3 \}) 
&= \mathbb{P}^1 - \{(\pm 6^3:1)\} \\
&= \PhiNew(\Gr(2,4)^s).
\end{align*}
Moreover, whenever $\rho = 0, \infty, \pm 1, \pm 3$, we have that
\begin{align*}
\PhiNew(p_{\rho}) 
&\in \{(\pm 6^3:1)\},
\end{align*}
implying that $p_{\rho}$ is not $\PGL_2$-stable by Theorem \ref{GIT-quotient-pencils-of-binary-cubics}.
\end{proof}

\begin{myrmk}
\label{rel-bw-Newstead-and-Wall-reps}

Generically, there exists six representatives of the form $p_{\rho}$ from Proposition \ref{Wall-reps-stable-orbits} for every $\PGL_2$-stable orbit in $\Gr(2,4)$. This can be shown as follows: define the map
\[
\PhiWall: \mathbb{P}^1 - \FWall \rightarrow \Gr(2,4)^s, (\rho:1) \mapsto p_{\rho},
\]
where $\FWall$ corresponds to the set of non-$\PGL_2$-stable $p_{\rho}$'s:
\[
\FWall := \{ (\rho:1) \in \mathbb{P}^1 \mid \rho = 0, \infty, \pm 1, \pm 3 \} \subset \mathbb{P}^1.
\]
Also, define the map
\[
f: \mathbb{P}^1 - \FWall \rightarrow \mathbb{P}^1 - \FNew,
(\rho:1) \mapsto \left( I'^3|_{p_{\rho}}: J|_{p_{\rho}} \right)
\]
sending every $(\rho:1)$ to the $\PGL_2$-invariants of the corresponding $p_{\rho}$. Then, we have a commutative diagram
\[
  \begin{tikzcd}
  \mathbb{P}^1 - \FWall \arrow{r}{\PhiWall} \arrow{d}{f}
  & \Gr(2,4)^s \arrow{ld}{\PhiNew} \\
  \mathbb{P}^1 - \FNew &
  \end{tikzcd}
\]
so it follows that $f = \PhiNew\circ\PhiWall$ is a generic 6-to-1 map from the expressions of $I'|_{p_{\rho}}$ and $J|_{p_{\rho}}$ in Equations \ref{I-prime-for-Wall-rep} and \ref{J-prime-for-Wall-rep}.

Moreover, we shall see in Proposition \ref{S4-action-on-X} that the map $\PhiWall$ is $S_4$-equivariant when $\mathbb{P}^1$ is endowed with the $S_4$-action induced by the unique two-dimensional irreducible representation of $S_4$. This shall be a crucial element for the proof of our first main theorem, Theorem \ref{main-thm-S4-symmetry-stable-locus}.
\end{myrmk}

\section{Statement of main results}
\label{statement-of-main-results}

Throughout this section, fix the subgroup $S_4:=N(D_4) \leq \PGL_2$ as the normalizer of the fixed dihedral group of order 4,
\[
D_4 := \left\langle
\begin{pmatrix}
1 & 0\\
0 & -1\\
\end{pmatrix},
\begin{pmatrix}
0 & -1\\
1 & 0\\
\end{pmatrix}
\right\rangle \leq {\PGL}_2.
\]
Our first main result shall be the construction of the isomorphism of quotient stacks from Equation \ref{key-quot-stack-isom}.

\begin{mythm}
\label{main-thm-S4-symmetry-stable-locus}

Let $\Gr(2,4)^s$ be the collection of all $\PGL_2$-stable points of $\Gr(2,4)$, and let $\FWall \subset \mathbb{P}^1$ be defined as in Remark \ref{rel-bw-Newstead-and-Wall-reps}. Also, endow $\mathbb{P}^1$ with the $S_4$-action induced by the unique two-dimensional irreducible representation of $S_4$. Then, the map $\PhiWall$ defined in Remark \ref{rel-bw-Newstead-and-Wall-reps} is $S_4$-equivariant and extends to a morphism of schemes
\[
\Phi: {\PGL}_2\times (\mathbb{P}^1 - \FWall) \rightarrow \Gr(2,4)^s, (A, (\rho:1)) \mapsto A\cdot p_{\rho},
\]
and $\Phi$ descends to a $PGL_2$-equivariant isomorphism of algebraic spaces
\[
{\PGL}_2\times_{S_4} (\mathbb{P}^1 - \FWall) \cong Gr(2,4)^s
\]
where the action of $S_4$ on ${\PGL}_2 \times (\mathbb{P}^1 - \FWall)$ is given by the diagonal action:
\[
S_4 \times ({\PGL}_2 \times (\mathbb{P}^1 - \FWall)) \rightarrow {\PGL}_2 \times (\mathbb{P}^1 - \FWall),
(\sigma, (A, (x:y))) \mapsto (A\sigma, \sigma^{-1}\cdot(x:y)).
\]
Hence, there exists an induced isomorphism of quotient stacks
\[
[({\PGL}_2\times_{S_4} (\mathbb{P}^1 - \FWall))/{\PGL}_2] \cong [\Gr(2,4)^s/{\PGL}_2],
\]
and it follows that
\[
[(\mathbb{P}^1 - \FWall)/S_4] \cong [\Gr(2,4)^s/{\PGL}_2]
\]
as quotient stacks defined over $\mathrm{Spec}(k)$.
\end{mythm}

\begin{myrmk}
At the fundamental level, the isomorphism of Equation \ref{key-quot-stack-isom} is constructed based on Proposition \ref{Wall-reps-stable-orbits}. One may wonder if there also exists an isomorphism of quotient stacks
\[
[\mathcal{U}/{\PGL}_2] \cong [\mathbb{P}^1/S_4]
\]
where $\mathcal{U} \subset \Gr(2,4)$ is the collection of all points in $\Gr(2,4)$ in a $\PGL_2$-orbit of some $p_{\rho}$, given that $S_4$ acts on the space of all $p_{\rho}$'s. Unfortunately, there is no such extension: in the proof of Proposition \ref{Wall-reps-stable-orbits}, we saw that the $p_{\rho}$'s with $(\rho:1)\in \FWall$ correspond to non-stable $\PGL_2$-orbits, and the $\PGL_2$-stabilizers of those $p_{\rho}$'s turn out to be $N(T_{\PGL_2})$ by direct calculations (see Proposition \ref{GIT-classification-all-PGL2-orbits-pencils-of-binary-cubics}). 

In this way, there is no direct way of extending the isomorphism of Equation \ref{key-quot-stack-isom} to cover the semistable locus $\Gr(2,4)^{ss}$.
\end{myrmk}

Our second main result is a presentation of the ring $A^*_{\PGL_2}(\Gr(2,4)^s)$, which we compute using the isomorphism of quotient stacks given by Equation \ref{key-quot-stack-isom}. The presentation is based on Lemma \ref{S4-equivar-Chow-ring-of-a-point} to be demonstrated in Section \ref{computation-equivar-Chow-ring-of-locus-of-stable-pts}:
\[
A^*_{S_4} = \ZZ[\alpha, \nu, \zeta_1, \eta]/(2\alpha=2\nu=4\zeta_1=3\eta=0, \alpha\nu^{j}=\alpha^{j+1}(\zeta_1+\alpha^2)^j : j \geq 1)
\]
where the degrees of the generators are given by:
\[
\alpha \in A^1_{S_4}, \zeta_1, \eta \in A^2_{S_4}, \nu \in A^3_{S_4}.
\]

\begin{mythm}
\label{main-thm-Chow-ring-stable-locus}

Denote the collection of all stable points in $\Gr(2,4)$ under the natural $\PGL_2$-action by $\Gr(2,4)^s\subset \Gr(2,4)$. The $\PGL_2$-equivariant Chow ring of $\Gr(2,4)^s$ is given by
\begin{align*}
A^*_{\PGL_2}(\Gr(2,4)^s)
&\cong A^*_{S_4}(\mathbb{P}^1 - \FWall) \\
&\cong A^*_{S_4}[\zeta] / (\zeta^2+c_1^{S_4}(V)\zeta+c_2^{S_4}(V), 3\zeta, 3c_1^{S_4}(V)),
\end{align*}
where $\FWall \subset \mathbb{P}^1$ is the finite subset as given in Theorem \ref{main-thm-S4-symmetry-stable-locus}, 
\[
\zeta= c_1^{S_4}(\mathcal{O}(1)) \in A^1_{S_4}(\mathbb{P}^1 - \FWall)
\]
is the (pullback of) the $S_4$-equivariant first Chern class of a hyperplane in $\mathbb{P}^1$, and $c_1^{S_4}(V)$ is the $S_4$-equivariant first Chern class of the unique two-dimensional irreducible representation $V$ of $S_4$.
\end{mythm}

We remark that we may eliminate the generators $\nu, \eta$ and $c_1^{S_4}(V)$ from the presentation of $A^*_{\PGL_2}(\Gr(2,4)^s)$ in Theorem \ref{main-thm-Chow-ring-stable-locus}: a simplified presentation is given in Proposition \ref{equivar-Chow-ring-of-coll-of-stable-orbits}.

\section{Describing the S\textsubscript{4}-symmetry on $\mathrm{Gr}$(2,4)\textsuperscript{s}}
\label{description-S4-symmetry-stable-locus}

\subsection{$\mathrm{Stab}_{\PGL_2}(p)$ for $p\in \Gr(2,4)^s$}
\label{isotropy-subgroups-of-each-orbit}

In this section, we determine the stabilizers $\mathrm{Stab}_{\PGL_2}(p)$ for points $p$ inside the locus of stable points $\Gr(2,4)^s$. This will be the first step for our proof of Theorem \ref{main-thm-S4-symmetry-stable-locus}. We shall need the following theorem due to Klein \cite{Kl56} classifying the finite subgroups of $\PGL_2$.

\begin{mythm}[\cite{Kl56}]
\label{struct-finite-subgps-PGL2}

Let $k$ be an algebraically closed field, and let $\omega_m \in k^*$ be a primitive $m$-th root of unity for $m\geq 1$. Then, a finite subgroup of $PGL_2(k)$ is isomorphic to one of the following:
\begin{enumerate}
\item A cyclic group $C_m= \left\langle \bigl(\begin{smallmatrix} 1 & 0\\0 & \omega_m \end{smallmatrix} \bigr) \right\rangle$ of order $m$.
\item A dihedral group $D_{2m}= \left\langle \bigl(\begin{smallmatrix} 1 & 0\\0 & \omega_m \end{smallmatrix} \bigr), \bigl(\begin{smallmatrix} 0 & -1\\1 & 0 \end{smallmatrix} \bigr) \right\rangle$ of order $2m$, $m\geq 2$.
\item The alternating group on 4 letters (or the tetrahedral group) $A_4$ of order 12:
\[
A_4 = C_3 \ltimes D_4 =
\left\langle \sigma_3 := 
\begin{pmatrix}
\omega_4 & -1 \\
\omega_4 & 1 \\
\end{pmatrix} \right\rangle \ltimes D_4
\]
\item The symmetric group on 4 letters (or the octahedral group) $S_4$ of order 24:
\begin{align*}
S_4 &= S_3 \ltimes D_4
= \left\langle 
\sigma_3= 
\begin{pmatrix}
\omega_4 & -1 \\
\omega_4 & 1 \\
\end{pmatrix},
\sigma_2=
\begin{pmatrix}
-1 & 1 \\
1 & 1 \\
\end{pmatrix}
\right\rangle \ltimes D_4, 
\end{align*}
\item The alternating group on 5 letters (or the icosahedral group) $A_5$ of order 60. (The presentation is omitted given that it will not be relevant for us.)
\end{enumerate}
Moreover, all of the above have exactly one conjugacy class in $PGL_2(k)$.
\end{mythm}

Using the fact that the eigenvectors of $C_m$ are the monomial basis elements of $\Sym^3 W = k[t_0,t_1]_3$, we observe $\mathrm{Stab}_{\PGL_2}(p)$ may only contain elements of order at most 3 whenever $p\in\Gr(2,4)$ is stable. We record this observation in the following Lemma.

\begin{mylemma}
\label{order-stabilizers-stable-orbits}

Let $A\in \PGL_2(k)$ and $p\in \Gr(2,4)^s$ be a pencil in a stable orbit, and suppose $A\cdot p=p$. Then, the order of $A$ is at most 3. It follows that $\mathrm{Stab}_{\PGL_2}(p)$ is a finite group with each of its elements of order at most 3.
\end{mylemma}
\begin{proof}
Firstly, the last statement is implied by the statement of $\mathrm{ord}(A) \leq 3$, since any pencil $p$ in a stable $\PGL_2$-orbit has $\mathrm{Stab}_{\PGL_2}(p)$ being a finite group by definition from basic GIT. 

Notice that stabilizers of pencils in the same $\PGL_2$-orbit are all conjugates to each other as subgroups of $\PGL_2$. By Theorem \ref{struct-finite-subgps-PGL2}, it suffices to determine the pencils that are fixed by elements of $\PGL_2(k)$ of the form
\[
X_n=
\begin{pmatrix}
1 & 0\\
0 & \omega_n\\
\end{pmatrix}
\]
where $\omega_n$ is a primitive $n$-th root of unity, and find the stable orbits containing such pencils. It is evident that the eigenvectors of $X_n$ are given by the monomial basis elements of $k[t_0,t_1]_3$, with eigenvalues
\[
\mu_i = \omega_n^i
\]
for each of $t_0^{3-i}t_1^i$, $i=0,\dots,3$. Thus, the eigenvectors of $X_n$ all have distinct eigenvalues whenever $n\geq 4$, so any $p\in \Gr(2,4)$ fixed by $X_n$ for any $n\geq 4$ must be spanned by some linearly independent monomials $f,g \in k[t_0,t_1]_3$. 

However, enumerating all the six possible $p=\langle f,g \rangle$ with $f,g$ distinct monomials, we may calculate that $(I')^3|_p=\pm 6^3 J|_p$, so $p$ is not stable by Theorem \ref{GIT-quotient-pencils-of-binary-cubics}. Therefore, $X_n\in \PGL_2(k)$ fixes a pencil $p$ in a stable orbit only if $n=1,2,3$, and the result follows.
\end{proof}

Next, using an elementary argument with linear algebra, we observe that there is only one stable orbit with isotropy subgroup containing elements of order 3. 

\begin{mylemma}
\label{stabilizer-stable-orbits-with-order-3-elems}

Let
\begin{align}
\label{stable-orbit-Z0}
\mathcal{O}_{A_4} := \{ p \in \Gr(2,4) \mid I'|_p = 0, J|_p \neq 0 \}
\end{align}
be the $\PGL_2$-orbit in $\Gr(2,4)$ consisting of points with invariants $I'=0$ and $J\neq 0$ (note that this is a unique orbit by Theorem \ref{GIT-quotient-pencils-of-binary-cubics}). Then, for any $p\in\Gr(2,4)^s$, $\mathrm{Stab}_{\PGL_2}(p)$ contains an element of order 3 if and only if $p\in \mathcal{O}_{A_4}$.
\end{mylemma}
\begin{proof}
We will find generators for all pencils in stable orbits fixed by
\[
X_3 =
\begin{pmatrix}
1 & 0\\
0 & \omega_3\\
\end{pmatrix}
\]
where $\omega_3$ is a primitive third root of unity. Let $p$ be a pencil in a stable orbit such that $X_3\cdot p = p$. Viewing $p$ as a 2-dimensional $X_3$-invariant subspace of $\Sym^3 W$, we may write $p=\langle f,g \rangle$ for linearly independent eigenvectors $f,g$ of $X_3$. By the proof of Lemma \ref{order-stabilizers-stable-orbits}, $p$ cannot be spanned by monomials only, so we must have that
\[
f = a_0t_0^3+a_3t_1^3 \in p, g = b_2t_0t_1^2+b_3t_1^3
\]
where $a_0,a_3 \neq 0$ and either $b_2\neq 0$ or $b_3\neq 0$ (but not both). Notice that all other possible combinations of eigenvectors $f,g$ could be row reduced to linearly independent monomials.

Thus, $p$ has invariants $I'=3p_{03}-p_{12}=0$ and $J\neq 0$ (where $p_{ij}=a_ib_j-a_jb_i$), so the result follows.
\end{proof}

Next, we may deduce $\mathrm{Stab}_{\PGL_2}(p)$ for $p\in \Gr(2,4)^s-\mathcal{O}_{A_4}$.

\begin{myprop}
\label{stabilizer-stable-orbits-not-Z0-expression}

Let $p\in \Gr(2,4)^s$ be $\PGL_2$-stable.
\begin{enumerate}
\item If $p\in \Gr(2,4)^s-\mathcal{O}_{A_4}$ (with $\mathcal{O}_{A_4}$ defined as in Equation \ref{stable-orbit-Z0}), then $\mathrm{Stab}_{\PGL_2}(p)\cong D_4 = \ZZ/2 \times \ZZ/2$.
\item If $p\in \mathcal{O}_{A_4}$, then $D_4 \leq \mathrm{Stab}_{\PGL_2}(p)$.
\end{enumerate}
\end{myprop}
\begin{proof}
Firstly, let $p\in \Gr(2,4)^s$ be stable. By Proposition \ref{Wall-reps-stable-orbits}, every $\PGL_2$-stable orbit contains at least one pencil of the form $p_{\rho} = \langle f_{\rho}, g_{\rho} \rangle = \langle t_0^3+\rho t_0t_1^2, \rho t_0^2t_1+t_1^3 \rangle$ for some $\rho\in k$. It is straightforward to verify that
\[
\begin{pmatrix}
1 & 0\\
0 & -1\\
\end{pmatrix},
\begin{pmatrix}
0 & -1\\
1 & 0\\
\end{pmatrix}
\]
fix every $p_{\rho}$, and the above elements generate the distinguished $D_4\leq \PGL_2$ as presented in Theorem \ref{struct-finite-subgps-PGL2}. Thus, $D_4\leq \mathrm{Stab}_{\PGL_2}(p)$ for every $p$ stable, and the second statement follows.

Now, suppose $p\in\Gr(2,4)^s-\mathcal{O}_{A_4}$. Lemma \ref{order-stabilizers-stable-orbits} and Lemma \ref{stabilizer-stable-orbits-with-order-3-elems} imply that $\mathrm{Stab}_{\PGL_2}(p)$ contains only elements of order 1 or 2. By Theorem \ref{struct-finite-subgps-PGL2}, the only finite subgroups of $\PGL_2$ with elements of order at most 2 are $\{I \}, \ZZ/2$ and $D_4$. Therefore, $\mathrm{Stab}_{\PGL_2}(p) \cong D_4$ whenever $p\in \Gr(2,4)^s-\mathcal{O}_{A_4}$.
\end{proof}

Finally, we may deduce that $\mathrm{Stab}_{\PGL_2}(p) \cong A_4$ for $p\in \mathcal{O}_{A_4}$ using Theorem \ref{struct-finite-subgps-PGL2}.

\begin{myprop}
\label{stabilizer-stable-orbit-Z0}

Let $p\in \mathcal{O}_{A_4}$. Then, $\mathrm{Stab}_{\PGL_2}(p) \cong A_4$, the alternating group on four letters.
\end{myprop}
\begin{proof}
Lemma \ref{stabilizer-stable-orbits-with-order-3-elems} and Lemma \ref{stabilizer-stable-orbits-not-Z0-expression} imply that $\mathrm{Stab}_{\PGL_2}(p)$ contains $D_4$ and an element of order 3 whenever $p\in \mathcal{O}_{A_4}$. Also, Lemma \ref{order-stabilizers-stable-orbits} implies that $\mathrm{Stab}_{\PGL_2}(p)$ cannot contain elements of order strictly greater than 3. Thus, by Theorem \ref{struct-finite-subgps-PGL2}, it follows that $\mathrm{Stab}_{\PGL_2}(p) \cong A_4$.
\end{proof}

\subsection{The $S_4$-symmetry on $\mathrm{Gr}(2,4)^s$}

For the rest of this paper, fix $D_4$ and $S_4 = N_{PGL_2}(D_4)$ to be the distinguished finite subgroups of $\PGL_2$ as given in Theorem \ref{struct-finite-subgps-PGL2}. We may specify the action of $S_4 := N(D_4)$ on the stable locus $\Gr(2,4)^s$.

\begin{mylemma}
\label{S4-symmetry-of-locus-of-stable-points}

Consider one of the ${\PGL}_2$-orbit representatives $p_{\rho}$ as given in Proposition \ref{Wall-reps-stable-orbits}, and suppose that $p_{\rho}\in \Gr(2,4)$ is ${\PGL}_2$-stable. Then, we have that
\[
N_{{\PGL}_2}(\mathrm{Stab}_{{\PGL}_2}(p_{\rho})) = S_4,
\]
and the orbit of the action of $S_4$ on $p_{\rho}$ is given by:
\[
S_4 \cdot p_{\rho} = \left\{ p_{\rho}, p_{-\rho}, p_{\frac{\rho+3}{\rho-1}}, p_{-\frac{\rho+3}{\rho-1}}, p_{\frac{\rho-3}{\rho+1}}, p_{-\frac{\rho-3}{\rho+1}} \right\}.
\]
Moreover, two ${\PGL}_2$-stable pencils $p_{\rho}, p_{\rho'}$ with $\rho,\rho'\in k\cup\{ \infty \}$ are in the same ${\PGL}_2$-orbit if and only if there exists some $\sigma\in S_4$ such that $\sigma\cdot p_{\rho}=p_{\rho'}$.
\end{mylemma}
\begin{proof}
From Proposition \ref{stabilizer-stable-orbits-not-Z0-expression}: we know that if $p_{\rho} \in \Gr(2,4)^s-\mathcal{O}_{A_4}$, then $\mathrm{Stab}_{{\PGL}_2}(p_{\rho}) = D_4$. It is a standard fact that $N_{{\PGL}_2}(D_4) = S_4$, and one may verify by direct calculations that the fixed $S_4$ with presentation given in Theorem \ref{struct-finite-subgps-PGL2} is indeed the normalizer of $D_4$. In addition, by direct computations:
\begin{align}
\sigma_3^{-1} \cdot p_{\rho} 
&= p_{-\frac{\rho+3}{\rho-1}}, \nonumber \\
\sigma_3 \cdot p_{\rho} &= p_{\frac{\rho-3}{\rho+1}}, \nonumber \\
\sigma_2 \cdot p_{\rho} &= p_{-\frac{\rho-3}{\rho+1}}, \nonumber \\
\sigma_3^{-1}\sigma_2 \cdot p_{\rho} &= p_{\frac{\rho+3}{\rho-1}}, \nonumber \\
\sigma_3\sigma_2 \cdot p_{\rho} &= p_{-\rho}, \label{action-outer-autom-gp-stabilizer-on-pencils}
\end{align}
where $\sigma_2, \sigma_3$ are generators of $S_4$ as given in Theorem \ref{struct-finite-subgps-PGL2}. Hence, the result follows for the case $p_{\rho}\in \Gr(2,4)^s-\mathcal{O}_{A_4}$.

Now, suppose $p_{\rho}\in \mathcal{O}_{A_4}$, so that 
\[
D_4 \leq \mathrm{Stab}_{{\PGL}_2}(p_{\rho}) \cong A_4
\]
by Proposition \ref{stabilizer-stable-orbits-not-Z0-expression}. It is straightforward to verify that $D_4$ is a normal subgroup of $\mathrm{Stab}_{{\PGL}_2}(p_{\rho})$, and that $\mathrm{Stab}_{{\PGL}_2}(p_{\rho})$ has the presentation
\[
\mathrm{Stab}_{{\PGL}_2}(p_{\rho}) = \langle \sigma_3 \rangle \ltimes D_4
\] 
where $\sigma_3$ is the order 3 generator of $A_4$ as given in Theorem \ref{struct-finite-subgps-PGL2}. Hence,
\[
N_{{\PGL}_2}(A_4) = S_4,
\]
and the $S_4$ action on $p_{\rho}\in \mathcal{O}_{A_4}$ is also given by Equations \ref{action-outer-autom-gp-stabilizer-on-pencils}.

Finally, suppose that two ${\PGL}_2$-stable pencils $p_{\rho}, p_{\rho'} \in \Gr(2,4)^s$ are in the same ${\PGL}_2$-orbit, i.e.,
\[
A\cdot p_{\rho} = p_{\rho'}
\]
for some $A\in {\PGL}_2$. The stabilizers of $p_{\rho}, p_{\rho'}$ must be equal in this case, so the above condition is either equivalent to
\[
A\in N(D_4) = S_4
\]
if $p_{\rho}, p_{\rho'}\notin \mathcal{O}_{A_4}$ or equivalent to
\[
A\in N(A_4) = S_4
\]
if $p_{\rho}, p_{\rho'}\in \mathcal{O}_{A_4}$. It follows that $p_{\rho}, p_{\rho'} \in \Gr(2,4)^s$ are in the same ${\PGL}_2$-orbit if and only if $\sigma\cdot p_{\rho} = p_{\rho'}$ for some $\sigma\in S_4$.
\end{proof}

Lastly, the $S_4$-symmetry on the collection of $\PGL_2$-stable $p_{\rho}$'s may be described concisely as an $S_4$-symmetry on $\mathbb{P}^1$ using the map $\PhiWall$ defined in Remark \ref{rel-bw-Newstead-and-Wall-reps}. This will be crucial for the proof of Theorem \ref{main-thm-S4-symmetry-stable-locus}.

\begin{myprop}
\label{S4-action-on-X}

Let $\PhiWall: \mathbb{P}^1 - \FWall \rightarrow \Gr(2,4)^s$ be the map defined in Remark \ref{rel-bw-Newstead-and-Wall-reps} given by
\[
\PhiWall(\rho:1) = p_{\rho}.
\]
Endow $\mathbb{P}^1$ with the $S_4$-action induced by its unique two-dimensional irreducible representation. Then, $\PhiWall$ is $S_4$-equivariant.
\end{myprop}
\begin{proof}
We will construct an $S_4$-action on $\mathbb{P}^1$ that is compatible with that on the $p_{\rho}$'s and then show that our construction agrees with the $S_4$-action induced by its unique two-dimensional irreducible representation. For each $\rho\in k\cup\{ \infty \}$ and each $\sigma\in S_4$, denote
\[
p_{\sigma\cdot \rho} := \sigma \cdot p_{\rho}
\] 
where $\sigma\cdot\rho$ is given by Equations \ref{action-outer-autom-gp-stabilizer-on-pencils}. Define an action of $S_4$ on $\mathbb{P}^1$ via
\[
\pi: S_4\times \mathbb{P}^1 \rightarrow \mathbb{P}^1, (\sigma, (\rho:1)) \mapsto (\sigma\cdot\rho:1).
\]
Hence, $\PhiWall$ is $S_4$-equivariant when $\mathbb{P}^1$ is endowed with the $S_4$-action given by $\pi$.

Now, using Equations \ref{action-outer-autom-gp-stabilizer-on-pencils}, we find that for $\sigma_3, \sigma_2 \in S_4$ as in Lemma \ref{S4-symmetry-of-locus-of-stable-points}:
\begin{align*}
\sigma_3 \cdot (\rho:1) &= \left( \frac{\rho-3}{\rho+1}: 1 \right), \\
\sigma_2 \cdot (\rho:1) &= \left( -\frac{\rho-3}{\rho+1}: 1 \right).
\end{align*}
Since the kernel of $\pi$ is $D_4$, and
\[
\Tr(\pi(\sigma_2)) = 0, \Tr(\pi(\sigma_3)) \in k^*
\]
whenever $k$ is of characteristic not equal to 2 or 3, we conclude that $\pi$ is induced by the unique 2-dimensional irreducible representation of $S_4$ by composing that representation with the quotient map from $k^2\setminus \{ (0,0) \}$ to $\mathbb{P}^1$.
\end{proof}

\begin{myrmk}
\label{rmk-S4-symm-observation-in-paper-Wa83}

The following $S_4$-symmetry on the subset of $\mathbb{P}(\Sym^4 W)$ consisting of every
\[
\Wr(p_{\rho}) = [\rho:0:3-\rho^2:0:\rho] = [1:0:3\rho^{-1}-\rho:0:1],  \text{ where }\rho\in k^*-\{ \pm 1, \pm 3\},
\] 
was observed by Wall in \cite{Wa83}: apply the change of variables $4\lambda - 2 := 3\rho^{-1}-\rho$ and set $f_{\lambda} := \Wr(p_{\rho})$, so that $S_4$ acts on the subset of all $f_{\lambda}$ with $\lambda\in k^* - \{1\}$. Then, the $S_4$-orbit of $f_{\lambda}$ is given by the action of the anharmonic subgroup of $\PGL_2$ on the index $\lambda$, i.e.,
\[
S_4\cdot f_{\lambda}
= \left\{ f_{\lambda}, f_{1-\lambda^{-1}}, f_{(1-\lambda)^{-1}}, f_{\lambda^{-1}}, f_{1-\lambda},  f_{\lambda(\lambda-1)^{-1}} \right\}.
\]
In this way, the $S_4$-symmetry on $\Gr(2,4)^s$ that we describe may be seen as lifting the $S_4$-symmetry from \cite{Wa83} over fibers of the Wronskian map. 

We also remark that it takes more work to establish our $S_4$-symmetry on $\Gr(2,4)^s$ directly from the one found by Wall on the $\Wr(p_{\rho})$'s \cite{Wa83}: there exists points such as $[1:0:0:0:1]\in \mathbb{P}(\Sym^4 W)$ that are fixed by a $D_8\leq S_4$ but whose fibers over $\Wr$ are only fixed by $D_4$, so these points make computations of the $S_4$-action on the $p_{\rho}$'s more complicated if we proceed from Wall's observation.

Finally, we also observe that the $\mathbb{P}^1$ appearing in Equation \ref{key-quot-stack-isom} is not identified with the 2-dimensional vector space $W=k[t_0,t_1]_1$ in any way: $W$ is not an $S_4$-representation, while $\mathbb{P}^1$ is induced by the unique 2-dimensional irreducible representation of $S_4$.
\end{myrmk}

\section{Proof of Theorem \ref{main-thm-S4-symmetry-stable-locus}}
\label{description-of-equivar-Chow-ring-of-locus-of-stable-pts-S4-symm}

Now, we are ready to prove Theorem \ref{main-thm-S4-symmetry-stable-locus}.
\begin{proof}[Proof of Theorem \ref{main-thm-S4-symmetry-stable-locus}]
Define the morphism of schemes over $\mathrm{Spec}(k)$
\[
\Phi: {\PGL}_2\times (\mathbb{P}^1 - \FWall) \rightarrow \Gr(2,4)^s, (A, (\rho:1)) \mapsto A\cdot p_{\rho}
\]
obtained by extending the map $\PhiWall: \mathbb{P}^1 - \FWall \rightarrow \Gr(2,4)^s$ to ${\PGL}_2\times (\mathbb{P}^1 - \FWall)$. We have already shown in Proposition \ref{S4-action-on-X} that $\PhiWall$ is $S_4$-equivariant, so $\Phi$ descends naturally to the quotient map
\[
{\PGL}_2\times_{S_4} (\mathbb{P}^1 - \FWall) \rightarrow \Gr(2,4)^s
\]
where the algebraic space ${\PGL}_2\times_{S_4} (\mathbb{P}^1 - \FWall) := ({\PGL}_2\times (\mathbb{P}^1 - \FWall))/S_4$ is defined via the diagonal action of $S_4$:
\[
S_4 \times ({\PGL}_2 \times (\mathbb{P}^1 - \FWall)) \rightarrow {\PGL}_2 \times (\mathbb{P}^1 - \FWall),
(\sigma, (A, (x:y))) \mapsto (A\sigma, \sigma^{-1}\cdot(x:y)).
\]

We shall use the following general fact from the theory of algebraic spaces to prove that $\Phi$ descends to an isomorphism from ${\PGL}_2\times_{S_4} (\mathbb{P}^1 - \FWall)$ to $\Gr(2,4)^s$.

\begin{description}
\item [Fact] (Proposition 5.2.5 of \cite{Ol16}) Let $X\rightarrow Y$ be an \'{e}tale surjective morphism of algebraic spaces over a base scheme $S$ with $X$ being a scheme over $S$. Then, $R := X\times_Y X$ is a scheme over $S$, and the inclusion $R \xhookrightarrow{} X\times_S X$ is an \'{e}tale equivalence relation. Moreover, the induced quotient map
\[
X/R \rightarrow Y
\]
is an isomorphism.
\end{description}

Hence, it suffices for us to show that:
\begin{description}
\item[(1)] $\Phi$ is $PGL_2$-equivariant, surjective and \'{e}tale;
\item[(2)] The equivalence relation on ${\PGL}_2\times (\mathbb{P}^1 - \FWall)$ given by
\[
R:= ({\PGL}_2\times (\mathbb{P}^1 - \FWall)) \times_{Gr(2,4)^s} ({\PGL}_2\times (\mathbb{P}^1 - \FWall))
\]
is the equivalence relation generated by the $S_4$-action on ${\PGL}_2\times (\mathbb{P}^1 - \FWall)$, so that
\[
({\PGL}_2\times (\mathbb{P}^1 - \FWall))/R = {\PGL}_2\times_{S_4} (\mathbb{P}^1 - \FWall).
\]
\end{description}

\begin{description}
\item[Proof of (1)] Firstly, we show that $\Phi$ is ${\PGL}_2$-equivariant. Let $A, B\in {\PGL}_2$, and let $\rho \in k\cup \{\infty\}$ be such that $\rho \neq 0,\infty,\pm 1, \pm 3$ so that $(\rho: 1)\in \mathbb{P}^1 - \FWall$. Notice that
\begin{align*}
\Phi(B \cdot (A, (\rho:1))) &:= \Phi(BA, (\rho:1)) \\
&= \langle f_{\rho}(A^{-1}B^{-1}(t_0,t_1)), g_{\rho}(A^{-1}B^{-1}(t_0,t_1)) \rangle \\
&= B\cdot \Phi(A, (\rho:1)),
\end{align*}
so it follows that $\Phi$ is a ${\PGL}_2$-equivariant morphism. Moreover, notice that $\Phi$ is surjective from Proposition \ref{Wall-reps-stable-orbits}. 

Next, in order to show that $\Phi$ is \'{e}tale, we will check that $\Phi$ is flat and unramified. These will be direct consequences of the following claim.
\begin{description}
\item [Claim] For every $p\in \Gr(2,4)^s$, the fiber $\Phi^{-1}(p)$ consists of exactly 24 points.
\item [Proof of Claim] Let $p=\langle f,g \rangle \in \Gr(2,4)^s$. By Proposition \ref{Wall-reps-stable-orbits}, we may fix some $\rho\in k\cup\{\infty\}$ with $(\rho:1)\in \mathbb{P}^1-\FWall$ and some $A\in {\PGL}_2$ such that $A\cdot p_{\rho}=p$. Hence,
\[
\Phi^{-1}(p) \supset X := \{ (A\tau\sigma, \sigma^{-1}\cdot (\rho:1)) \mid \tau\in \mathrm{Stab}_{{\PGL}_2}(p_{\rho}), \sigma \in S_4  \}
\]
where $|X|=|S_4|=24$, so it suffices for us to show the reverse inclusion. From Lemma \ref{S4-symmetry-of-locus-of-stable-points}, $p_{\rho'}$ is in the same ${\PGL}_2$-orbit as $p_{\rho}$ if and only if $\sigma\cdot p_{\rho}=p_{\rho'}$ for some $\sigma\in S_4$. It follows that
\[
(B, (\rho':1)) \in \Phi^{-1}(p) \implies (\rho':1) = \sigma^{-1}\cdot (\rho:1)
\]
for some $\sigma \in S_4$, so
\begin{align*}
B\sigma^{-1}\cdot p_{\rho} = A \cdot p_{\rho}
&\iff A^{-1}B\sigma^{-1}\cdot p_{\rho} = p_{\rho} \\
&\iff A^{-1}B\sigma^{-1}= \tau \text{ for some }\tau \in \mathrm{Stab}_{{\PGL}_2}(p_{\rho}), \\
&\iff B = A\tau\sigma \text{ for some }\tau \in \mathrm{Stab}_{{\PGL}_2}(p_{\rho}).
\end{align*}
Thus, the claim follows.
\end{description}

It follows that $\Phi$ is a finite morphism with equi-dimensional (all finite) fibers. Since both the source and the target of $\Phi$ are regular schemes over $\mathrm{Spec}(k)$, $\Phi$ is flat by miracle flatness. The fact that $\Phi$ is unramified follows from $\Phi^{-1}(p)$ being a discrete set of 24 points for every geometric point $p\in \Gr(2,4)^s$. Hence, $\Phi$ is an \'{e}tale morphism.

\item[Proof of (2)] Let
\[
R:= ({\PGL}_2\times (\mathbb{P}^1 - \FWall)) \times_{Gr(2,4)^s} ({\PGL}_2\times (\mathbb{P}^1 - \FWall)).
\]
In the context of the \'{e}tale surjective morphism $\Phi$, we have that
\begin{align*}
((A, (\rho:1)), (B, (\rho':1))) \in R
&\iff A \cdot p_{\rho} = B\cdot p_{\rho'}\\
&\iff B = A\tau\sigma,
\end{align*}
 for some $\tau\in \mathrm{Stab}_{{\PGL}_2}(p_{\rho})$ and $\sigma\in S_4$ such that $\sigma\cdot p_{\rho'} = p_{\rho}$ by our previous claim, so $R$ is the equivalence relation on ${\PGL}_2\times (\mathbb{P}^1 - \FWall)$ given by the diagonal $S_4$-action. It follows that $\Phi$ induces a ${\PGL}_2$-equivariant isomorphism of algebraic spaces
\[
{\PGL}_2\times_{S_4} (\mathbb{P}^1 - \FWall) \xrightarrow{\cong} Gr(2,4)^s
\]
by universality of factorization through quotient maps.
\end{description}
\end{proof}

The isomorphism of quotient stacks given by Equation \ref{key-quot-stack-isom} readily follows from basic facts of algebraic stacks, which we record as the following Corollary.

\begin{mycor}
\label{struct-stable-locus-quotient-stacks}

Let $\Gr(2,4)^s$ and $\FWall \subset \mathbb{P}^1$ be as previously defined, and let $S_4 = N_{{\PGL}_2}(D_4)$ act on $\mathbb{P}^1$ as in Proposition \ref{S4-action-on-X}. Also, let ${\PGL}_2$ act on ${\PGL}_2\times_{S_4} (\mathbb{P}^1 - \FWall)$ and on $\Gr(2,4)^s$ as in Theorem \ref{main-thm-S4-symmetry-stable-locus}. Then,
\[
[({\PGL}_2\times_{S_4} (\mathbb{P}^1 - \FWall))/{\PGL}_2] \cong [\Gr(2,4)^s/{\PGL}_2]
\]
as quotient stacks over $\mathrm{Spec}(k)$. It follows that
\[
[(\mathbb{P}^1 - \FWall)/S_4] \cong [\Gr(2,4)^s/{\PGL}_2]
\]
as quotient stacks over $\mathrm{Spec}(k)$.
\end{mycor}
\begin{proof}
By universality of factorization through quotient maps, we have a 2-commuting Cartesian diagram
\[
  \begin{tikzcd}
  {\PGL}_2\times_{S_4} (\mathbb{P}^1 - \FWall) \arrow{r}{\cong} \arrow{d}{} 
  & Gr(2,4)^s \arrow{d}{} \\
  {[({\PGL}_2\times_{S_4} (\mathbb{P}^1 - \FWall))/{\PGL}_2]}  \arrow{r}{} & {[\Gr(2,4)^s/{\PGL}_2]}
  \end{tikzcd}
\]
where the vertical arrows are the respective quotient maps. Hence, the bottom arrow induced by the top isomorphism must also be an isomorphism of quotient stacks. 

Now, since $S_4\leq PGL_2$ is a closed subgroup:
\[
[({\PGL}_2\times_{S_4} (\mathbb{P}^1 - \FWall))/{\PGL}_2] \cong [(\mathbb{P}^1 - \FWall)/S_4]
\]
as quotient stacks, so the last isomorphism follows.
\end{proof}

\begin{myrmk}
\label{rmk-char-2-or-3-case-main-thm}

Given that our proof of Theorem \ref{main-thm-S4-symmetry-stable-locus} is written in the framework of algebraic spaces, it would be reasonable to conjecture that Theorem \ref{main-thm-S4-symmetry-stable-locus} also holds in characteristics 2 or 3, i.e., over any algebraically closed field regardless of the characteristic. However, at the fundamental level, our proof is constructed using the Wall representatives as well as the $S_4$-equivariant map $\PhiWall: \mathbb{P}^1 - \FWall \rightarrow \Gr(2,4)^s$ defined in Remark \ref{rel-bw-Newstead-and-Wall-reps}, which rely on the assumptions made in \cite{Ne81} that the characteristic of the base field is not equal to 2 or 3 . Constructing a proof of Theorem \ref{main-thm-S4-symmetry-stable-locus} in characteristic 2 or 3 will likely involve a different and much more lengthy argument, so we choose to omit discussions for the characteristic 2 or 3 cases.
\end{myrmk}

\section{Proof of Theorem \ref{main-thm-Chow-ring-stable-locus}}
\label{computation-equivar-Chow-ring-of-locus-of-stable-pts}

By Corollary \ref{struct-stable-locus-quotient-stacks} and results on Chow rings of quotient stacks from Section 5 of \cite{EG98}:
\[
A^*_{S_4}(\mathbb{P}^1 - \FWall)\cong A^*_{\PGL_2}(\Gr(2,4)^s),
\]
since $\PGL_2$ acts smoothly on $\Gr(2,4)^s$ giving rise to the smooth quotient stack $[\Gr(2,4)^s/\PGL_2]$. We will use excision to calculate $A^*_{S_4}(\mathbb{P}^1 - \FWall)$. The required exact sequences are the following:
\begin{align}
A^*_{S_4}(F_1) 
&\rightarrow A^*_{S_4}(\mathbb{P}^1)
\rightarrow A^*_{S_4}(\mathbb{P}^1 - F_1)
\rightarrow 0, \label{excision-seq-X1-in-P1} \\
A^*_{S_4}(F_2) 
&\rightarrow A^*_{S_4}(\mathbb{P}^1 - F_1)
\rightarrow A^*_{S_4}(\mathbb{P}^1 - \FWall)
\rightarrow 0, \label{excision-seq-X2-in-P1}
\end{align}
where the (disjoint) $S_4$-invariant subsets $F_1,F_2\subset \mathbb{P}^1$ are given by
\begin{align}
F_1 &= \{ (\rho:1) \mid \rho = \pm 1, \infty \} = \{ (1: 1), (-1:1), (1:0) \} \subset \mathbb{P}^1, \label{X1-for-orbit-Z2}\\
F_2 &= \{ (\rho:1) \mid \rho = 0, \pm 3 \} = \{ (0:1), (3:1), (-3:1) \} \subset \mathbb{P}^1, \label{X2-for-orbit-Z4}
\end{align}
so that $F_1 \cup F_2 = \FWall$. 

For the rest of this Section, denote $q_1:=(1:0)$ and $q_2:= (0:1)$. From direct calculations:
\begin{align}
\mathrm{Stab}_{S_4}(q_1) &= \mathrm{Stab}_{S_4}(q_2) = \langle D_4, \sigma_3^{-1}\sigma_2 \rangle \cong D_8 \leq S_4, \label{S4-stabilizer-X1-X2}
\end{align}
where $\sigma_3, \sigma_2$ are given as in Theorem \ref{struct-finite-subgps-PGL2}. Thus,
\[
A^*_{S_4}(F_1) \cong A^*_{S_4}(F_2)  \cong A^*_{D_8}(\{ \text{pt} \}).
\]
Hence, in order to obtain the presentation for $A^*_{S_4}(\mathbb{P}^1 - \FWall)$ given in Theorem \ref{main-thm-Chow-ring-stable-locus}, it suffices for us to compute the following:
\begin{enumerate}
\item A presentation for $A^*_{S_4}(\mathbb{P}^1)$: this follows directly from the projective bundle formula after we use a result by Burt Totaro \cite{To99} to find an explicit presentation for $A^*_{S_4}$, which we record as Lemma \ref{S4-equivar-Chow-ring-of-a-point};
\item The images of the pushforward maps on the left-hand side of the exact sequences \ref{excision-seq-X1-in-P1} and \ref{excision-seq-X2-in-P1}: these are computed by Lemma \ref{excision-D8-equivar-Chow-X1-X2-from-P1} and Lemma \ref{transfer-hyperpl-and-chern-classes-D8-to-S4-over-P1}, whose proofs are essentially lengthy but standard arguments via diagram chasing.
\end{enumerate}

\begin{mylemma}
\label{S4-equivar-Chow-ring-of-a-point}

The $S_4$-equivariant Chow ring of a point is given by
\[
A^*_{S_4} = \ZZ[\alpha, \nu, \zeta_1, \eta]/(2\alpha=2\nu=4\zeta_1=3\eta=0, \alpha\nu^{j}=\alpha^{j+1}(\zeta_1+\alpha^2)^j : j \geq 1)
\]
where the degrees of the generators are given by:
\[
\alpha \in A^1_{S_4}, \zeta_1, \eta \in A^2_{S_4}, \nu \in A^3_{S_4}.
\]
\end{mylemma}
\begin{proof}
Denote the classifying space of a group $G$ by $BG$. Lemma 10.1 of \cite{To99} states that the $p$-localization of the Chow group of $BS_n$ for any $n\geq 1$ is given by
\[
A^*(BS_n)_{(p)} = A^*(BH_p) \cap H^*(BS_n, \ZZ)_{(p)} \subset H^*(BH_p, \ZZ)_{(p)}
\]
where $H_p \leq S_n$ is a Sylow $p$-subgroup of $S_n$ for $p$ prime. \footnote{Here, we remark that Totaro denotes the Chow ring of the classifying space of a group $G$ by $CH^*BG$ and the \textit{operational} Chow ring of a space $X$ by $A^*(X)$. These rings coincide whenever the spaces considered are smooth, so we always use the notation $A^*(-)$ given that all spaces under our consideration are smooth orbits under linear algebraic group actions.}

The cohomology ring of $BS_4$ with integer coefficients is given by C. B. Thomas \cite{Th74}:
\begin{align*}
H^*(BS_4, \ZZ) = \ZZ[\alpha', \nu', \zeta'_1, \eta'] / &(2\alpha'=2\nu'=4\zeta'_1=3\eta'=0, \\
&\alpha'\nu'^{2j}=\alpha'^{j+1}(\zeta'_1+\alpha'^2)^j : j \geq 1)
\end{align*}
where the degrees of the generators are given by
\[
\mathrm{deg}(\alpha') = 2, \mathrm{deg}(\nu') = 3, \mathrm{deg}(\zeta'_1) = \mathrm{deg}(\eta') = 4.
\]
In this way, $H^*(BS_4, \ZZ)_{(2)}$ is generated by $\alpha', \nu', \zeta'_1$ and $H^*(BS_4, \ZZ)_{(3)}$ is generated by $\eta'$, with all other $p$-localizations of $H^*(BS_4, \ZZ)$ trivial. 

From Proposition 12.1 of \cite{To99}, the 2-torsion part of $A^i(BS_4)$ is isomorphic to $H^{2i}(BS_4, \ZZ/2)$ for each $i\geq 0$. Moreover, each Sylow 3-subgroup of $S_4$ is isomorphic to $\ZZ/3$, whose Chow ring is generated by a 3-torsion element of degree 1. Thus, using Lemma 10.1 of \cite{To99} as given above:
\[
A^*(BS_4)_{(3)} = A^*(B\ZZ/3) \cap H^*(BS_4, \ZZ)_{(3)}
\cong H^{2*}(BS_4, \ZZ)_{(3)} = \ZZ[\eta']/(3\eta').
\]
The result follows by denoting the images of $\alpha', \nu'^2, \zeta'_1, \eta'$ under the above isomorphisms by $\alpha, \nu, \zeta_1, \eta \in A^*(BS_4)$ respectively.
\end{proof}

Hence, using the projective bundle formula (for example, from Chapter 2 of \cite{AF24}), we obtain that
\[
A^*_{S_4}(\mathbb{P}^1) \cong
A^*_{S_4}[\zeta]/(\zeta^2+c_1^{S_4}(V)\zeta+c_2^{S_4}(V))
\]
where $\zeta = c_1^{S_4}(\mathcal{O}(1))$ is the $S_4$-equivariant first Chern class of a hyperplane in $\mathbb{P}^1$, and $V$ is the $S_4$-representation induced by the $S_4$-action on $\mathbb{P}^1$. In particular, $V$ is the unique 2-dimensional irreducible representation of $S_4$ as described in Remark \ref{S4-action-on-X}.

\begin{myprop}
\label{img-of-maps-f1-f2-from-X1-X2-to-P1}

The ring $A^*_{S_4}(\mathbb{P}^1-\FWall)$ is given by
\[
A^*_{S_4}(\mathbb{P}^1-\FWall) 
\cong A^*_{S_4}(\mathbb{P}^1) / (3\zeta, 3c_1^{S_4}(V)).
\]
\end{myprop}
\begin{proof}
Denote $q_1 = (1:0), q_2 = (0:1)$, so that
\[
F_i = S_4 \times_{D_8} \{ q_i \}
\]
for $i=1,2$ with $D_8 = \mathrm{Stab}_{S_4}(\{ q_i \})$ fixed. Consider the commutative diagrams of proper $S_4$-equivariant morphisms
\[
  \begin{tikzcd}
  F_i = S_4 \times_{D_8} \{ q_i \} \arrow{r}{\tilde{f_i}} \arrow{d}{\varphi_i} 
  & S_4\times_{D_8} \mathbb{P}^1 = \mathaccent\cdot\cup_{j=1}^3 \mathbb{P}^1 \arrow{d}{\tilde{\varphi}} \\
  F_i  \arrow{r}{f_i} & \mathbb{P}^1
  \end{tikzcd}
\]
for $i=1,2$, where $f_i: F_i \xhookrightarrow{} \mathbb{P}^1$ are inclusion maps, $\varphi_i$ are given by the natural $S_4$-equivariant map
\[
\varphi_i: (\sigma, q_i) \mapsto \sigma\cdot q_i,
\]
$\tilde{f_i}$ is induced by $f_i$ (equivalently, by the inclusion $\{ q_i \} \xhookrightarrow{} \mathbb{P}^1$):
\[
\tilde{f_i}: (\sigma, q_i) \mapsto (\sigma, q_i),
\]
and $\tilde{\varphi}$ is the 3-to-1 map from $S_4\times_{D_8} \mathbb{P}^1$ to $\mathbb{P}^1$ induced by $\varphi_i$ (both $\varphi_1, \varphi_2$ induce the same $\tilde{\varphi}$ here):
\[
\tilde{\varphi}: (\sigma, (x:y)) \mapsto \sigma\cdot(x:y).
\]
Thus, we get a commutative diagrams of pushforwards on the respective $S_4$-equivariant Chow rings:
\[
  \begin{tikzcd}
  A^*_{S_4}(S_4 \times_{D_8} \{ q_i \}) \cong A^*_{D_8} \arrow{r}{\tilde{f_i}_*} \arrow{d}{\varphi_{i*}} 
  & A^*_{S_4}(S_4 \times_{D_8} \mathbb{P}^1) \cong A^*_{D_8}(\mathbb{P}^1) \arrow{d}{\tilde{\varphi}_{*}} \\
  A^*_{S_4}(F_i)  \arrow{r}{f_{i*}} & A^*_{S_4}(\mathbb{P}^1)
  \end{tikzcd}
\]
for $i=1,2$. It is clear that $\varphi_i$ is homotopic to the identity map on $F_i$ for $i=1,2$ respectively, so $\varphi_{i*}$ is an isomorphism for $i=1,2$. Also, from the definition of $\tilde{\varphi}$, $\tilde{\varphi}_{*}$ is the transfer homomorphism from the $D_8$-equivariant Chow ring of $\mathbb{P}^1$ to the $S_4$-equivariant Chow ring of $\mathbb{P}^1$ for $i=1,2$. It follows that
\[
\mathrm{im}(f_{i*}) = \mathrm{im}(\tilde{\varphi}_{*} \circ \tilde{f_{i}}_*)
\]
for $i=1,2$. 

Now, by functoriality, we know that the pullbacks $f_i^*$, $i=1,2$, restrict to the restriction homomorphism 
\[
\mathrm{res}^{S_4}_{D_8}: A^*_{S_4} \rightarrow A^*_{D_8}
\]
on the subring $A^*_{S_4} \subset A^*_{S_4}(\mathbb{P}^1)$. From Chapter 13 of \cite{To14}, $A^*_{D_8}$ is generated by Chern classes of $D_8$-representations. Since $D_8$ is a Sylow 2-subgroup of $S_4$, $\mathrm{res}^{S_4}_{D_8}$ is surjective via restricting $S_4$-representations to $D_8$-representations:
\[
\mathrm{res}^{S_4}_{D_8}(c_i^{S_4}(W)) = c_i^{D_8}(W \downarrow^{S_4}_{D_8})
\]
for any $W$ an $S_4$-representation.

Thus, for any $u\in A^*_{D_8}$ of positive degree, there exists some $u'\in A^*_{S_4}$ such that $f_i^*(u')=u$. By the projection formula (for example, from Chapter 3 of \cite{AF24}):
\[
f_{i*}(u) = f_{i*}(f_i^*(u') \cdot 1) = u' \cdot f_{i*}(1)
\]
for $i=1,2$. It follows that $f_{i*}(1)$ generates the image of $f_{i*}$ for $i=1,2$ as ideals in $A^*_{S_4}(\mathbb{P}^1)$.

Finally, we observe that $f_{1*}$ is the pushforward map on the left-hand side of the exact sequence \ref{excision-seq-X1-in-P1} by definition, so
\[
A^*_{S_4}(\mathbb{P}^1 - F_1) \cong A^*_{S_4}(\mathbb{P}^1) / \mathrm{im}(f_{1*}) = A^*_{S_4}(\mathbb{P}^1) / (f_{1*}(1))
\]
by exactness. For the pushforward map on the left-hand side of the exact sequence \ref{excision-seq-X2-in-P1}, note that we have a fiber square of $S_4$-equivariant morphisms
\[
  \begin{tikzcd}
  F_2 = (\mathbb{P}^1-F_1) \times_{\mathbb{P}^1} F_2 \arrow{r}{id_{F_2}} \arrow{d}{f} 
  & F_2 \arrow{d}{f_2} \\
  \mathbb{P}^1-F_1 \arrow{r}{g} & \mathbb{P}^1
  \end{tikzcd}
\]
where $f,g$ are inclusion maps with $f$ proper, given that $F_2$ is closed in $\mathbb{P}^1-F_1$. By definition, the pushforward of $f$,
\[
f_*: A^*_{S_4}(F_2) \rightarrow A^*_{S_4}(\mathbb{P}^1-F_1)
\]
is the pushforward on the left-hand side of \ref{excision-seq-X2-in-P1}. Then, by naturality of Gysin homomorphisms (Theorem 6.2 of \cite{Fu84}, cf. Chapter 3 of \cite{AF24}):
\[
f_*(u) = f_*(id_{F_2}^*(u)) = g^*(f_{2*}(u))
\]
for any $u \in A^*_{S_4}(F_2)$. It follows that
\begin{align*}
A^*_{S_4}(\mathbb{P}^1 - (F_1\cup F_2)) 
&\cong A^*_{S_4}(\mathbb{P}^1 -F_1) / \mathrm{im}(f_*) \\
&\cong A^*_{S_4}(\mathbb{P}^1 - F_1) / (g^*(f_{2*}(1))), \\
&\cong (A^*_{S_4}(\mathbb{P}^1)/ (f_{1*}(1)) / (g^*(f_{2*}(1))),
\end{align*}
where the second-last line follows from the exactness of \ref{excision-seq-X2-in-P1} implying the surjectivity of $g^*$, so that $\mathrm{im}(f_*) = g^*(\mathrm{im}(f_{2*})) = (g^*(f_{2*}(1)))$.

The Proposition then follows from Lemma \ref{excision-D8-equivar-Chow-X1-X2-from-P1} and Lemma \ref{transfer-hyperpl-and-chern-classes-D8-to-S4-over-P1}, which give
\begin{align*}
&f_{1*}(1) = 3\zeta \in A^*_{S_4}(\mathbb{P}^1), \\
&g^*(f_{2*}(1)) = 3c_1^{S_4}(V) \in A^*_{S_4}(\mathbb{P}^1 - F_1).
\end{align*}
\end{proof}

\begin{mylemma}
\label{excision-D8-equivar-Chow-X1-X2-from-P1}

Fix $D_4$ to be as in Theorem \ref{struct-finite-subgps-PGL2} and fix $D_8 = \mathrm{Stab}_{S_4}(q_i)$ to be as in Equation \ref{S4-stabilizer-X1-X2} for $i=1,2$. Denote the $D_8$-equivariant first Chern class of a hyperplane in $\mathbb{P}^1$ by $\xi = c_1^{D_8}(\mathcal{O}_{\mathbb{P}^1}(1))$, and let $\tilde{f_{1}}_*, \tilde{f_{2}}_*$ be as in Proposition \ref{img-of-maps-f1-f2-from-X1-X2-to-P1}. Then: 
\begin{enumerate}
\item $[\{ q_1 \}]^{D_8} = \tilde{f_{1}}_*(1) = \xi$.
\item $[\{ q_2 \}]^{D_8} = \tilde{f_{2}}_*(1) = \xi + c_{1, D_4}$ where $c_{1,D_4}:=c_1^{D_8}(k_{D_4})$ is the $D_8$-equivariant first Chern class of $k_{D_4}$, the unique 1-dimensional irreducible representation of $D_8$ with kernel $D_4 \leq D_8$.
\end{enumerate}
\end{mylemma}
\begin{proof}
Our goal is to calculate the classes $[\{ q_i \}]^{D_8} \in A^*_{D_8}(\mathbb{P}^1)$, where
\[
[\{ q_i \}]^{D_8} = \tilde{f_{i}}_*(1)
\]
from the definition of $\tilde{f_{i}}_*$ for $i=1,2$ in Proposition \ref{img-of-maps-f1-f2-from-X1-X2-to-P1}. By the projective bundle formula \cite{AF24}:
\[
A^*_{D_8}(\mathbb{P}^1) \cong A^*_{D_8}[\xi]/(\xi^2 + c_1^{D_8}(V \downarrow^{S_4}_{D_8})\xi + c_2^{D_8}(V \downarrow^{S_4}_{D_8})),
\]
where $V$ is the unique 2-dimensional irreducible representation of $S_4$. Since all of our calculations are in $A^*_{D_8}(\mathbb{P}^1)$, denote $V = V \downarrow^{S_4}_{D_8}$. To establish the results, we will show:
\begin{description}
\item[Claim 1] $ch^{D_8}(V) = 1 + c_{1,D_4} \in A^*_{D_8}$;
\item[Claim 2] $\tilde{f_1}^*(\xi) = - c_{1,D_4}$ and $\tilde{f_2}^*(\xi) = 0$.
\end{description} 
Then, we will use the self-intersection formula
\[
\tilde{f_i}^*\circ \tilde{f_{i}}_*(1) = c_{top}^{D_8}(T_{q_i} \mathbb{P}^1)
\]
to calculate $\tilde{f_{i}}_*(1)$ for $i=1,2$.

\begin{description}
\item [Proof of Claim 1] We may establish this claim using the character table of $D_8$ over a field of characteristic not equal to 2. Fix $D_4, D_8$ as described. Denote the character of $k_{D_4}$ by $\chi_{D_4}$ and the character of $V=V \downarrow^{S_4}_{D_8}$ by $\chi_V$ (where both $k_{D_4}$ and $V$ are considered as $D_8$-representations). Then, by computing the inner product of $\chi_V$ with each character in the character table of $D_8$, we obtain that
\[
\chi_V = \chi_{triv} + \chi_{D_4},
\]
where $\chi_{triv}$ is the character of the trivial representation. It follows that
\[
ch^{D_8}(V) = ch^{D_8}(k_{triv}\oplus k_{D_4}) = 1 + c_{1,D_4},
\]
as required.

\item [Proof of Claim 2] By functoriality,
\[
\tilde{f_i}^*(\xi) = c_1^{D_8}(\tilde{f_i}^* \mathcal{O}_{\mathbb{P}^1}(1))
\]
for $i=1,2$. Since $\tilde{f_i}: S_4 \times_{D_8} \{ q_i \} \rightarrow S_4 \times_{D_8} \mathbb{P}^1$ is induced by the inclusion map $\{ q_i \} \xhookrightarrow{} \mathbb{P}^1$ for $i=1,2$, $\tilde{f_i}^*$ restricts $S_4$-equivariant (therefore $D_8$-equivariant) vector bundles over $\mathbb{P}^1$ to vector bundles over $q_i$:
\begin{align*}
\tilde{f_i}^* \mathcal{O}_{\mathbb{P}^1}(1)
&= \mathcal{O}_{\mathbb{P}^1}(1) |_{q_i} 
= L_{q_i}^*,
\end{align*}
the dual space of the line spanned by $q_i$. Thus,
\[
\tilde{f_1}^*(\xi) = c_1^{D_8}(L_{q_1}^*) = - c_1^{D_8}(L_{q_1}) = -c_{1,D_4}
\]
given that $D_8 = \langle D_4, \sigma_3^{-1}\sigma_2 \rangle$ acts on the line spanned by $q_1=(1,0)$ (as a vector in $k^2$ here) via 
\[
\sigma_3^{-1}\sigma_2\cdot q_1 = (-1,0), D_4\cdot q_1 = q_1,
\]
i.e., via the 1-dimensional representation $k_{D_4}$. Similarly,
\[
\tilde{f_2}^*(\xi) = c_1^{D_8}(L_{q_2}^*) = - c_1^{D_8}(L_{q_2}) = 0
\]
since 
\[
\sigma_3^{-1}\sigma_2 \cdot q_2 = q_2, D_4 \cdot q_2 = q_2.
\]
This concludes the proof of Claim 2.
\end{description}

Next, using the fact that the tangent bundle of $\mathbb{P}^1$ is given by
\[
\mathcal{T}_{\mathbb{P}^1} \cong \mathcal{H}om(\mathcal{S}, \mathcal{Q})
\]
where $\mathcal{Q}$ is the universal quotient bundle, we may calculate $\tilde{f_i}^*\circ \tilde{f_{i}}_*(1)$ for $i=1,2$ using the self-intersection formula:
\begin{align*}
\tilde{f_i}^*\circ \tilde{f_{i}}_*(1) &= c_{top}^{D_8}(T_{q_i} \mathbb{P}^1) \\
&= c_1^{D_8}(\mathrm{Hom}_k(\mathcal{S}|_{q_i}, \mathcal{Q}|_{q_i})), \\
&= c_1^{D_8}(L_{q_i}^* \otimes (V/L_{q_i})) \\
&= c_1^{D_8}(L_{q_i}^*) + c_1^{D_8}(V / L_{q_i}).
\end{align*}
Hence,
\begin{align}
\label{pullback-class-of-q1-to-point}
\tilde{f_1}^*([q_1]^{D_8}) &= \tilde{f_1}^*\circ \tilde{f_{1}}_*(1) = -c_{1,D_4}
\end{align}
since
\[
ch^{D_8}(V/ L_{q_1}) = \frac{ch^{D_8}(V)}{ch^{D_8}(L_{q_1})} = \frac{1+c_{1,D_4}}{1+c_{1,D_4}} = 1.
\]
Similarly,
\begin{align}
\label{pullback-class-of-q2-to-point}
\tilde{f_2}^*([q_2]^{D_8}) &= \tilde{f_2}^*\circ \tilde{f_{2}}_*(1) = c_1^{D_8}(V / L_{q_2}) = c_{1,D_4}
\end{align}
since
\[
ch^{D_8}(V / L_{q_2}) = \frac{ch^{D_8}(V)}{ch^{D_8}(L_{q_2})} = \frac{1+c_{1,D_4}}{1} = 1+c_{1,D_4}.
\]

Now, by naturality, $\{q_1\} \cap \{q_2\} = \varnothing$ implies that
\begin{align}
\label{pullback-q1-q2-point-naturality}
\tilde{f_1}^*([q_2]^{D_8}) &= 0 = \tilde{f_2}^*([q_1]^{D_8}).
\end{align}
Hence, Claim 2 implies that $[q_1]^{D_8}, [q_2]^{D_8} \in A^*_{D_8}(\mathbb{P}^1)$ are both linear combinations of $\xi$ and $c_{1,D_4}$, as $f_1^*, f_2^*$ restrict to the identity map on $A^*_{D_8}$ by functoriality (as in the proof of Proposition \ref{img-of-maps-f1-f2-from-X1-X2-to-P1}). This means we may write
\begin{align*}
[q_1]^{D_8} &= a_1\xi + a_2 c_{1,D_4}, \\
[q_2]^{D_8} &= b_1\xi + b_2 c_{1,D_4},
\end{align*}
so solving for the coefficients via by Equations \ref{pullback-class-of-q1-to-point}, \ref{pullback-class-of-q2-to-point} and \ref{pullback-q1-q2-point-naturality}, we have
\[
[q_1]^{D_8} = \xi, [q_2]^{D_8} = \xi + c_{1,D_4},
\]
as claimed.
\end{proof}

\begin{mylemma}
\label{transfer-hyperpl-and-chern-classes-D8-to-S4-over-P1}

Let
\[
\tilde{\varphi}_*: A^*_{S_4}(S_4 \times_{D_8} \mathbb{P}^1) \rightarrow  A^*_{S_4}(\mathbb{P}^1), 
\]
be the pushforward map induced by
\[
\tilde{\varphi_i}: S_4 \times_{D_8} \mathbb{P}^1 \rightarrow \mathbb{P}^1, (\sigma, (x:y)) \mapsto \sigma\cdot (x:y)
\]
as in Proposition \ref{img-of-maps-f1-f2-from-X1-X2-to-P1} for $i=1,2$. Then:
\begin{enumerate}
\item $\tilde{\varphi}_*(\xi) = 3\zeta$;
\item $\tilde{\varphi}_*(c_{1,D_4}) = 3\zeta + 3c_1^{S_4}(V)$.
\end{enumerate}
\end{mylemma}
\begin{proof}
Similarly as in the proof of Lemma \ref{excision-D8-equivar-Chow-X1-X2-from-P1}, we will first calculate pullbacks of classes under
\[
\tilde{\varphi}^*: A^*_{S_4}(\mathbb{P}^1) \rightarrow A^*_{S_4}(S_4 \times_{D_8} \mathbb{P}^1) \cong A^*_{D_8}(\mathbb{P}^1)
\]
which is the restriction morphism from $A^*_{S_4}(\mathbb{P}^1)$ to $A^*_{D_8}(\mathbb{P}^1)$ by the definition of $\tilde{\varphi}$ in Proposition \ref{img-of-maps-f1-f2-from-X1-X2-to-P1}. Then, we will use the projection formula to find the classes $\tilde{\varphi}_*(\xi), \tilde{\varphi}_*(c_{1,D_4})$.

Since $\tilde{\varphi}^*$ restricts to the restriction homomorphism $\mathrm{res}^{S_4}_{D_8}: A^*_{S_4} \rightarrow A^*_{D_8}$ on the subring $A^*_{S_4} \subset A^*_{S_4}(\mathbb{P}^1)$, we have that $\tilde{\varphi}^*(c_j^{S_4}(W)) = c_j^{D_8}(W \downarrow^{S_4}_{D_8})$ by functoriality. In Lemma \ref{excision-D8-equivar-Chow-X1-X2-from-P1}, we showed that
\[
ch^{D_8}(V^{S_4}_{D_8}) = 1 + c_{1,D_4}
\]
with $V$ being the unique 2-dimensional irreducible representation of $S_4$. Thus,
\[
\tilde{\varphi}^*(c_1^{S_4}(V)) = c_{1,D_4}.
\]

Next, notice that
\[
\tilde{\varphi}^*(\zeta) = \tilde{\varphi}^*(c_1^{S_4}(\mathcal{O}(1)))
= c_1^{S_4}(\tilde{\varphi}^*\mathcal{O}(1)) \\
= c_1^{S_4}(S_4\times_{D_8} \mathcal{O}(1))
\]
from the definition of $\tilde{\varphi_i}: S_4\times_{D_8} \mathbb{P}^1 \rightarrow \mathbb{P}^1$. Thus,
\[
\tilde{\varphi}^*(\zeta) = c_1^{S_4}(S_4\times_{D_8} \mathcal{O}(1))
= c_1^{D_8}(\mathcal{O}(1)) = \xi
\]
from the natural identification
\[
A^*_{S_4}(S_4\times_{D_8} \mathbb{P}^1) \cong A^*_{D_8}(\mathbb{P}^1).
\]

Now, since $\tilde{\varphi}$ is proper and $\tilde{\varphi}(S_4\times_{D_8} \mathbb{P}^1) = \mathbb{P}^1$, the image of the fundamental class under $\tilde{\varphi}_*$ is
\[
\tilde{\varphi}_*(1) = \mathrm{deg}(\tilde{\varphi}) \cdot [\mathbb{P}^1]^{S_4}
= 3 \in A^*_{S_4}(\mathbb{P}^1).
\]

Therefore, by the projection formula:
\begin{align*}
\tilde{\varphi}_*(\xi) &= \tilde{\varphi}_*(\tilde{\varphi}^*(\zeta) \cdot 1)
= \zeta \cdot \tilde{\varphi}_*(1) = 3\zeta, \\
\tilde{\varphi}_*(c_{1,D_4}) &= \tilde{\varphi}_*(\tilde{\varphi}^*(c_1^{S_4}(V)) \cdot 1)
= c_1^{S_4}(V) \cdot \tilde{\varphi}_*(1) = 3c_1^{S_4}(V),
\end{align*}
as claimed.
\end{proof}

Lemma \ref{excision-D8-equivar-Chow-X1-X2-from-P1} and Lemma \ref{transfer-hyperpl-and-chern-classes-D8-to-S4-over-P1} give
\begin{align*}
f_{1*}(1) &= \tilde{\varphi}_*\circ \tilde{f}_{1*}(1) = \tilde{\varphi}_*(\xi) = 3\zeta,\\
f_{2*}(1) &= \tilde{\varphi}_*\circ \tilde{f}_{2*}(1) = \tilde{\varphi}_*(\xi+c_{1,D_4}) = 3\zeta+3c_1^{S_4}(V),
\end{align*}
so this finishes the proof of Proposition \ref{img-of-maps-f1-f2-from-X1-X2-to-P1}.

Finally, using results from \cite{Th74} and \cite{Ev65}, we may eliminate some of the generators of $A^*_{\PGL_2}(\Gr(2,4)^s) \cong A^*_{S_4}(\mathbb{P}^1-\FWall)$ coming from the subring $A^*_{S_4} \subset A^*_{S_4}(\mathbb{P}^1-\FWall)$.
\begin{myprop}
\label{equivar-Chow-ring-of-coll-of-stable-orbits}

The $\PGL_2$-equivariant Chow ring of $\Gr(2,4)^s$ is given by
\begin{align*}
A^*_{\PGL_2}(\Gr(2,4)^s)
&\cong A^*_{S_4}(\mathbb{P}^1-\FWall) \\
&\cong A^*_{S_4}[\zeta] / (\zeta^2+c_1^{S_4}(V)\zeta+c_2^{S_4}(V), 3\zeta, 3c_1^{S_4}(V)) \\
&\cong \ZZ[\alpha, \zeta_1, \zeta]/(2\alpha=4\zeta_1=3\zeta=0, \alpha^2=0)
\end{align*}
where
\[
\alpha\in A^1_{S_4}, \zeta_1 \in A^2_{S_4},
\]
are generators of $A^*_{S_4}$ given as in Proposition \ref{S4-equivar-Chow-ring-of-a-point}, and
\[
\zeta= c_1^{S_4}(\mathcal{O}(1)) \in A^1_{S_4}(\mathbb{P}^1-\FWall)
\]
is the (pullback of) $S_4$-equivariant first Chern class of a hyperplane in $\mathbb{P}^1$.
\end{myprop}
\begin{proof}
From Lemma 10.1 of \cite{To99}, there exists a split injection from the Chow ring of $S_n$ to the integral cohomology ring of $S_n$ for all $n\geq 1$. Thus, every relation between elements of the subring
\[
A^*_{S_4} \subset A^*_{S_4}(\mathbb{P}^1-\FWall) \cong A^*_{S_4}[\zeta] / (\zeta^2+c_1^{S_4}(V)\zeta+c_2^{S_4}(V), 3\zeta, 3c_1^{S_4}(V))
\]
maps injectively to the integral cohomology ring of $S_4$. 

Hence, to simplify the expression of $A^*_{S_4}(\mathbb{P}^1-\FWall)$ given by Proposition \ref{img-of-maps-f1-f2-from-X1-X2-to-P1}, we may introduce the relation
\[
3c_1^{S_4}(V) = 0
\]
to the cohomology ring of $S_4$ \cite{Th74}:
\begin{align*}
H^*(S_4, \ZZ) = \ZZ[\alpha', \nu', \zeta'_1, \eta'] / &(2\alpha'=2\nu'=4\zeta'_1=3\eta'=0, \\
&\alpha'\nu'^{2j}=\alpha'^{j+1}(\zeta'_1+\alpha'^2)^j : j \geq 1).
\end{align*}
Then, we may recover $A^*_{S_4}(\mathbb{P}^1-\FWall)$ by mapping $H^*_{S_4}/(3c_1^{S_4}(V))$ back to $A^*_{S_4} \subset A^*_{S_4}(\mathbb{P}^1-\FWall)$ via
\[
\alpha' \mapsto \alpha, \nu'^2 \mapsto \nu, \zeta_1 \mapsto \zeta_1, \eta' \mapsto \eta,
\]
as shown in Lemma \ref{S4-equivar-Chow-ring-of-a-point} using Lemma 10.1 and 12.1 of \cite{To99}.

In order to understand the relation $3c_1^{S_4}(V)=0$ in $H^*_{S_4}$, we need to take a closer look at the 2-primary part of $H^*_{S_4}$, which maps one-to-one into the cohomology ring of a Sylow 2-subgroup of $S_4$ under restriction (\cite{Th74}, also cf. \cite{To99}). $H^*_{D_8}$ with coefficients in $\ZZ$ is given by (\cite{Ev65}, as in \cite{Th74}):
\[
H^*_{D_8} \cong \ZZ[\alpha', \beta', \nu', \zeta''_1]/(2\alpha'=2\beta'=2\nu'=4\zeta_1''=0, \alpha'^2=\alpha'\beta', \nu'^2=\beta'\zeta_1''),
\]
where $\alpha', \nu' \in H^*_{D_8}$ are the restrictions of their counterparts $\alpha', \nu' \in H^*_{S_4}$ under $\mathrm{res}^{S_4}_{D_8}$, and
\[
\mathrm{res}^{S_4}_{D_8}(\zeta'_1) = \zeta''_1 + \beta'^2
\]
with $\beta'\in H^2_{D_8}$, $\zeta''_1\in H^4_{D_8}$.

From \cite{Ev65}, the generators $\alpha', \beta', \zeta''_1 \in H^*_{D_8}$ are Chern classes of the $D_8$-representations
\[
\alpha' = c_1^{D_8}(k_{\langle \sigma_4 \rangle}),
\beta' = c_1^{D_8}(k_{D_4}) = c_{1,D_4},
\zeta''_1 = c_1^{D_8}(k^2),
\]
using notations from Lemma \ref{excision-D8-equivar-Chow-X1-X2-from-P1}, where $k_{\langle \sigma_4 \rangle}$ is the one-dimensional irreducible representation of $D_8$ with kernel $\langle \sigma_4 \rangle \cong \ZZ/4 \leq D_8$, and $k^2$ is the unique two-dimensional irreducible representation of $D_8$. Thus, 
\[
\mathrm{res}^{S_4}_{D_8}(c_1^{S_4}(V)) = c_{1,D_4} = \beta'.
\]
Hence, we obtain $3\beta'=0$ by mapping the relation $3c_1^{S_4}(V)=0$ into $H^*_{D_8}$, so $\beta'=0$ in $H^*_{D_8}/(3\beta')$ implies that
\[
\alpha'^2=\nu'^2=0 \text{ in }H^*_{D_8}/(3\beta'),
\]
which we may lift injectively to $H^*_{S_4}/(3c_1^{S_4}(V))$. Moreover, the injective lifting of $\beta'=0$ gives
\[
c_1^{S_4}(V) = 0 \text{ in }H^*_{S_4}/(3c_1^{S_4}(V)).
\]
Thus, lifting back to $A^*_{S_4}(\mathbb{P}^1-\FWall)$, we obtain $\alpha^2=\nu=0$ and $c_1^{S_4}(V)=0$, i.e.,
\begin{align*}
A^*_{S_4}(\mathbb{P}^1-\FWall) &\cong A^*_{S_4}[\zeta]/(\zeta^2+c_2^{S_4}(V)=0, 3\zeta=0, \alpha^2=\nu=c_1^{S_4}(V)=0) \\
&\cong \ZZ[\alpha, \zeta_1, \eta, \zeta]/(2\alpha=4\zeta_1=3\eta=3\zeta=0, \alpha^2=0, \zeta^2+c_2^{S_4}(V)=0).
\end{align*}
Lastly, notice that
\[
\mathrm{res}^{S_4}_{D_8}(c_2^{S_4}(V)) = c_2^{D_8}(V \downarrow^{S_4}_{D_8}) = c_2^{D_8}(k_{triv}\oplus k_{D_4}) = 0
\]
as in Lemma \ref{excision-D8-equivar-Chow-X1-X2-from-P1}, so $c_2^{S_4}(V)$ must be in the 3-primary part of $A^*_{S_4}$ generated by $\eta$. By considering the restriction of $V$ to a Sylow 3-subgroup $C_3 \leq S_4$ and following the argument of \cite{Th74}, we have that
\[
\mathrm{res}^{S_4}_{C_3}(c_2^{S_4}(V)) = c_2^{C_3}(V \downarrow^{S_4}_{C_3}) = c_2^{C_3}(V_{std}) = \eta,
\]
where $V_{std}$ is the standard 2-dimensional representation of $C_3$. Thus, $c_2^{S_4}(V)=\eta\in A^*_{S_4}(\mathbb{P}^1-\FWall)$ as well, and it follows that
\begin{align*}
A^*_{S_4}(\mathbb{P}^1-\FWall) &\cong \ZZ[\alpha, \zeta_1, \eta, \zeta]/(2\alpha=4\zeta_1=3\eta=3\zeta=0, \alpha^2=0, \zeta^2+\eta=0) \\
&\cong \ZZ[\alpha, \zeta_1, \zeta]/(2\alpha=4\zeta_1=3\zeta=0, \alpha^2=0)
\end{align*}
as claimed.
\end{proof}

\section{Discussions and future directions}
\label{section-future-directions}

\subsection{Interpreting Theorem \ref{main-thm-Chow-ring-stable-locus}}
As before, let $k$ be an algebraically closed field of characteristic not equal to 2 or 3 so that Theorem \ref{main-thm-Chow-ring-stable-locus} holds. In the following, we give interpretations of the generators of $A^*_{\PGL_2}(\Gr(2,4)^s)$ from Theorem \ref{main-thm-Chow-ring-stable-locus} as elements of $A^*_{\PGL_2}(\Gr(2,4))$. Since the stable locus $\Gr(2,4)^s$ is open in $\Gr(2,4)$, we have a pullback homomorphism of Chow rings
\[
i^*: A^*_{\PGL_2}(\Gr(2,4)) \rightarrow A^*_{\PGL_2}(\Gr(2,4)^s)
\]
induced by the inclusion map $i: \Gr(2,4)^s \xhookrightarrow{} Gr(2,4)$. We shall use $i^*$ to construct lifts of the generators of $A^*_{\PGL_2}(\Gr(2,4)^s)$.

\begin{myprop}
\label{lift-of-class-of-hyperplane-section-to-entire-Chow-ring}

Let
\[
\zeta = c_1^{S_4}(\mathcal{O}(1)) 
\in A^1_{S_4}(\mathbb{P}^1-F) \cong A^1_{\PGL_2}(\Gr(2,4)^s)
\]
be the degree 1 generator of $A^*_{\PGL_2}(\Gr(2,4)^s)$ given by the pullback of the $S_4$-equivariant first Chern class of a hyperplane in $\mathbb{P}^1$ as in Theorem \ref{main-thm-Chow-ring-stable-locus}. Also, let $[\overline{\mathcal{O}_{A_4}}]^{\PGL_2} \in A^*_{\PGL_2}(\Gr(2,4))$ be the $\PGL_2$-equivariant class of the Zariski closure of the stable orbit $\mathcal{O}_{A_4}$ given by Equation \ref{stable-orbit-Z0}. Then:
\[
i^*[\overline{\mathcal{O}_{A_4}}]^{\PGL_2} = \zeta.
\]
\end{myprop}
\begin{proof}
Recall that $\mathcal{O}_{A_4} \subset \Gr(2,4)^s$ is the unique stable orbit with isotropy subgroup $A_4$ from Proposition \ref{stabilizer-stable-orbit-Z0}. By calculating the $I'$- and $J$-invariants of $\mathcal{O}_{A_4}$ using Equations \ref{I-prime-for-Wall-rep} and \ref{J-prime-for-Wall-rep}, we find that the unique ${\PGL}_2$-representatives of the form $p_{\rho} = \langle t_0^3+\rho t_0t_1^2, \rho t_0^2t_1+t_1^3 \rangle$ are the ones with $\rho=\pm i\sqrt{3}$. Thus, by Theorem \ref{main-thm-S4-symmetry-stable-locus}:
\[
\Phi({\PGL}_2\times \{ [i\sqrt{3}:1] \}) = \Phi({\PGL}_2\times \{ [-i\sqrt{3}:1] \}) = \mathcal{O}_{A_4},
\]
and we get the identification
\[
\mathcal{O}_{A_4} \cong {\PGL}_2\times_{S_4} \{ [i\sqrt{3}:1], [-i\sqrt{3}:1] \}
\]
via the induced ${\PGL}_2$-equivariant isomorphism of algebraic spaces. Notice that the $S_4$-equivariant hyperplane bundle $\mathcal{O}(1)$ over $\mathbb{P}^1$ can be identified with the $S_4$-invariant subset $\{ [i\sqrt{3}:1], [-i\sqrt{3}:1] \} \subset \mathbb{P}^1$ using the definition
\[
\mathbb{P}^1 := \mathrm{Proj}(V)
\]
where $V$ is the unique 2-dimensional irreducible representation of $S_4$. Hence,
\[
i^*([\overline{\mathcal{O}_{A_4}}]^{{\PGL}_2}) = [\mathcal{O}_{A_4}]^{{\PGL}_2} = c_1^{S_4}(\mathcal{O}(1)),
\]
where $[\mathcal{O}_{A_4}]^{{\PGL}_2} \in A^*_{{\PGL}_2}(\Gr(2,4)^s)$ is interpreted as the class of $\mathcal{O}_{A_4}$ in $\Gr(2,4)^s$.
\end{proof}

\begin{myrmk}
\label{Z0-is-the-only-nonreduced-orbit}

It is known that the Zariski closure of the stable orbit $\mathcal{O}_{A_4}$ is the only $\PGL_2$-orbit closure of $\Gr(2,4)$ that is non-reduced as a scheme over $\mathrm{Spec}(k)$. This can be shown by noticing that
\[
\overline{\mathcal{O}_{A_4}} = V(I'^3) \subset \Gr(2,4),
\]
from the definition of $\overline{\mathcal{O}_{A_4}}$ by Equation \ref{stable-orbit-Z0}. Thus, the above Proposition demonstrates that the  class $[\mathcal{O}_{A_4}]^{\PGL_2} \in A^*_{\PGL_2}(\Gr(2,4)^s)$ has a special meaning as well.
\end{myrmk}

The generators of $A^*_{\PGL_2}(\Gr(2,4)^s)$ coming from $A^*_{S_4}$ can be interpreted as restrictions of generators of $A^*_{\PGL_2}$. For simplicity, we will only give interpretations for generators appearing in the simplified presentation from Proposition \ref{equivar-Chow-ring-of-coll-of-stable-orbits}.

\begin{myprop}
\label{lift-of-class-of-S4-generators-to-entire-Chow-ring}

Let
\[
\alpha\in A^1_{S_4}, \zeta_1\in A^2_{S_4},
\]
be the generators of $A^*_{\PGL_2}(\Gr(2,4)^s)$ coming from $A^*_{S_4}$ as in Proposition \ref{equivar-Chow-ring-of-coll-of-stable-orbits}. Then:
\begin{align*}
i^*(c_2(\mathfrak{sl}_2)) &= \zeta_1,\\
i^*(c_3(\mathfrak{sl}_2)) &= \alpha\zeta_1,
\end{align*}
where $c_2(\mathfrak{sl}_2), c_3(\mathfrak{sl}_2)$ are the generators of $A^*_{\PGL_2}$ from Theorem \ref{integral-Chow-ring-of-classifying-space-of-PGL2}.
\end{myprop}
\begin{proof}
We have a commutative diagram
\[
  \begin{tikzcd}
  A^*_{\PGL_2} \arrow{r}{} \arrow{d}{} & A^*_{S_4} \arrow{d}{} \\
  A^*_{\PGL_2}(\Gr(2,4)^s)  \arrow{r}{\cong} & A^*_{S_4}(\mathbb{P}^1-\FWall)
  \end{tikzcd}
\]
where the downward arrows are the respective inclusions and the top arrow is the pullback homomorphism induced by restricting $\PGL_2$-representations to $S_4$-representations. Thus, it suffices for us to compute the restriction of $\mathfrak{sl}_2$ as a $\PGL_2$-representation to $S_4$. In fact, by setting $D_8 \leq S_4$ to be the same Sylow 2-subgroup of $S_4$ as in Lemma \ref{excision-D8-equivar-Chow-X1-X2-from-P1}, we may calculate directly using the character table of $D_8$ that
\[
\mathfrak{sl}_2 \downarrow^{\PGL_2}_{D_8} \cong k_{\langle \sigma_4 \rangle} \oplus k^2
\]
as $D_8$-representations, where $k_{\langle \sigma_4 \rangle}$ is the 1-dimensional $D_8$-representation with kernel $\langle \sigma_4 \rangle = C_4 \leq D_8$ and $k^2$ is the standard 2-dimensional irreducible $D_8$-representation. By keeping track of $A^*_{D_8}$ as the 2-primary part of $A^*_{S_4}$ as in the proof of Proposition \ref{equivar-Chow-ring-of-coll-of-stable-orbits}, we obtain the desired results.
\end{proof}

\subsection{Difficulties with computing $A^*_{\mathrm{PGL}_2}(\mathrm{Gr}(2,4))$}
\label{difficulties-entire-Chow-ring}

In the following, we present a more detailed discussion of the difficulties associated with calculating $A^*_{\PGL_2}(\Gr(2,4))$ using the stratification method. We follow the steps as explained in Section \ref{section-intro}, 
and begin by giving a natural $\PGL_2$-stratification of $\Gr(2,4)$ following \cite{Ne81}. 

\begin{myprop}[Section 4 of \cite{Ne81}]
\label{GIT-classification-all-PGL2-orbits-pencils-of-binary-cubics}

Let $p\in \Gr(2,4)$ be such that $p$ is not $\PGL_2$-stable, i.e., $p\notin Gr(2,4)^s$. Fix $T = T_{\PGL_2}$ to be the maximal torus of diagonal matrices in $\PGL_2$, and let $N(T) := N_{\PGL_2}(T_{\PGL_2})$ be its normalizer, i.e., $N(T) = S_2 \ltimes T$ where $S_2$ is the group of $2\times 2$ permutation matrices. Also, denote the Borel subgroup of upper triangular matrices of $\PGL_2$ by $B_2$. Then, $p$ belongs to one of the unique $\PGL_2$-orbits
\[
Z_1, Z^{(0)}_2, Z^{(1)}_2, Z^{(2)}_2, Z^{(1)}_3, Z^{(2)}_3
\]
as given in the below table, where $\dim Z^{(j)}_i = i$ for each $Z^{(j)}_i$ (the dimensions, $\PGL_2$-representatives and GIT-classifications of each of the orbits are also given in \cite{Ne81}):
\begin{center} 
\begin{tabular}{|c|c|c|c|c|}
\hline
Orbit & Representative $p^{(j)}_i$ & Isotropy subgroup & GIT-classification & $\dim Z^{(j)}_i$ \\ 
\hline
$Z_1$ & $\langle t_0^3, t_0^2t_1 \rangle$ 
           & $B_2$ & unstable & 1  \\ \hline
$Z^{(0)}_2$ & $\langle t_0^3, t_0t_1^2 \rangle$ 
           & $T$ & unstable & 2  \\ \hline
$Z^{(1)}_2$ &  $\langle t_0t_1^2, t_0^2t_1 \rangle$ 
           & $N(T)$ 
           & semistable & 2  \\ \hline
$Z^{(2)}_2$ &  $\langle t_0^3, t_1^3 \rangle$
           & $N(T)$ & semistable & 2  \\ \hline
$Z^{(1)}_3$ &  $\langle t_0t_1^2, t_0^2(t_0+t_1) \rangle$ 
           & $\ZZ/2$ & semistable & 3 \\ \hline 
$Z^{(2)}_3$ &  $\langle t_0^3, t_1^2(t_0+t_1) \rangle$
           & $\ZZ/2$  & semistable & 3 \\ \hline
\end{tabular}
\end{center}
\end{myprop}
\begin{proof}
The orbits $Z^{(j)}_i$ as well as their respective representatives $p^{(j)}_i$ are calculated by Newstead in Section 4 of \cite{Ne81}: these are also shown to be an exhaustive list of the non-stable $\PGL_2$-orbits of $\Gr(2,4)$. The expressions for isotropy subgroups may be found by direct computations.
\end{proof}

Next, we give explicit expressions for the Zariski closures of the $Z^{(j)}_i$'s.
\begin{myprop}
\label{closure-properties-of-Zi-in-Grassmannian}

Let $Z_1, Z^{(0)}_2, Z^{(1)}_2, Z^{(2)}_2, Z^{(1)}_3, Z^{(2)}_3$ be the non-stable $\PGL_2$-orbits as given in Proposition \ref{GIT-classification-all-PGL2-orbits-pencils-of-binary-cubics}, and denote their Zariski closures by $\overline{Z^{(j)}_i}$. Then:
\begin{enumerate}[label={\alph*.}]
\item $\overline{Z_1} = Z_1$.
\item $\overline{Z^{(0)}_2} = Z^{(0)}_2\cup Z_1$.
\item $\overline{Z^{(2)}_2} = Z^{(2)}_2\cup Z_1$.
\item $\overline{Z^{(2)}_3} = Z^{(2)}_3\cup Z^{(2)}_2\cup Z^{(0)}_2\cup Z_1$.
\item $\overline{Z^{(1)}_2} = Z^{(1)}_2\cup Z_1$.
\item $\overline{Z^{(1)}_3}= Z^{(1)}_3\cup Z^{(1)}_2\cup Z^{(0)}_2\cup Z_1$.
\end{enumerate}
\end{myprop}
\begin{proof}
We start from the orbit with the smallest dimension, $\dim Z_1 = 1$ from \cite{Ne81}, and work our way up. 
\begin{enumerate}[label={\alph*.}]
\item The statement $\overline{Z_1} = Z_1$ is trivial since the closure cannot contain other orbits of dimensions strictly greater than $Z_1$. 
\item From Proposition \ref{GIT-classification-all-PGL2-orbits-pencils-of-binary-cubics}:
\[
Z^{(0)}_2 \cup Z_1 = \{ p\in \Gr(2,4) \mid I'|_p = J|_p = 0 \}
\]
is the collection of all unstable points. It follows that the ideal of $Z^{(0)}_2\cup Z_1$, $(I', J)$, is prime since
\[
\codim(Z^{(0)}_2\cup Z_1) = \dim \Gr(2,4) - \dim(Z^{(0)}_2\cup Z_1) 
=\dim \Gr(2,4) - \dim Z^{(0)}_2 = 2,
\]
where $\dim Z^{(0)}_2=2$ is given by \cite{Ne81}. Hence, $Z^{(0)}_2\cup Z_1$ is irreducible, so $\overline{Z^{(0)}_2} = Z^{(0)}_2\cup Z_1$.
\item From Proposition \ref{GIT-classification-all-PGL2-orbits-pencils-of-binary-cubics}, $p^{(2)}_2 := \langle t_0^3,t_1^3 \rangle$ is a representative for the orbit $Z^{(2)}_2$. By expressing points in $Z^{(2)}_2$ as $A\cdot p^{(2)}_2$ for some $A\in \PGL_2$ and finding relations on their Pl\"{u}cker coordinates, one may use any computer algebra system (such as Macaulay2) to check that points on $Z^{(2)}_2$ satisfy
\begin{align*}
p_{12}(9p_{03}+p_{12}) &= 9p_{02}p_{13},
&p_{12}^2 = 9p_{01}p_{23}
\end{align*}
along with the Pl\"{u}cker relation, where $p_{ij}$'s are the Pl\"{u}cker coordinates. Moreover, using Macaulay2 \cite{M2}, one may also check that the above relations plus the Pl\"{u}cker relation give a prime ideal of codimension 2 in the homogeneous coordinate ring of the Pl\"{u}cker coordinates $p_{ij}$, so they define the closure of $Z^{(2)}_2$ in $\Gr(2,4)$. We may further check that $Z_1$ satisfies the above relations, so it follows that
\[
\overline{Z^{(2)}_2} = Z^{(2)}_2\cup Z_1
\]
given that every orbit closure consists of the orbit itself and other orbits of strictly smaller dimensions.
\item From Section 4 of Newstead \cite{Ne81}, we have that
\[
Z^{(2)}_3 \cup Z^{(2)}_2 = \phi^{-1}(-6^3:1),
\]
so $Z^{(2)}_2$ is in the orbit closure of $Z^{(2)}_3$. Moreover, $Z^{(2)}_3$ is the orbit of pencils with exactly one triple point from Proposition \ref{GIT-classification-all-PGL2-orbits-pencils-of-binary-cubics} using descriptions in Section 4 of \cite{Ne81}, so its orbit closure must contain all pencils with at least one triple point. Hence, $\overline{Z^{(0)}_2} \subset \overline{Z^{(2)}_3}$, so the result follows.
\item From Proposition \ref{GIT-classification-all-PGL2-orbits-pencils-of-binary-cubics}, $p^{(1)}_2=\langle t_0^2t_1, t_0t_1^2 \rangle$ is a representative of the orbit $Z^{(1)}_2$. Using Macaulay2 \cite{M2}, one may check that Pl\"{u}cker coordinates of all points on $Z^{(1)}_2$ satisfy:
\begin{align*}
p_{02}^2-p_{01}p_{03}-p_{01}p_{12}&=0, &p_{02}p_{03}-p_{01}p_{13}=0, \\
p_{13}^2-p_{03}p_{23}-p_{12}p_{23}&=0, &p_{03}p_{13}-p_{02}p_{23}=0, \\
p_{01}p_{23}-p_{03}^2 &=0, 
\end{align*}
along with the Pl\"{u}cker relation. One may also verify that these relations give a prime ideal of codimension 2 in the homogeneous coordinate ring of the Pl\"{u}cker coordinates, so $\overline{Z^{(1)}_2}$ is contained in the subvariety given by this ideal. Finally, one may check that all points on $Z_1$ satisfy the above relations, so $Z_1\subset \overline{Z^{(1)}_2}$ and the result follows.
\item Similar to the case of $\overline{Z^{(2)}_3}$, Section 4 of \cite{Ne81} gives
\[
Z^{(1)}_3 \cup Z^{(1)}_2 = \phi^{-1}(6^3:1),
\]
so $Z^{(1)}_2\subset \overline{Z^{(1)}_3}$. Also, $\overline{Z^{(0)}_2} \subset \overline{Z^{(1)}_3}$ since $\overline{Z^{(1)}_3}$ contains all pencils with at least one base point. Thus, the result follows.
\end{enumerate}
\end{proof}

In this way, we may summarize inclusion relations between the closed $\PGL_2$-strata in $\Gr(2,4) - \Gr(2,4)^s$ as in Figure \ref{fig-PGL2-non-stable}.
\begin{figure}
\[
\Gr(2,4) - \Gr(2,4)^s = \left\{
  \begin{tikzcd}
  & \overline{Z^{(1)}_2} \arrow[hookrightarrow]{r}{} & \overline{Z^{(1)}_3} \\
  \overline{Z_1} \arrow[hookrightarrow]{ru}{} \arrow[hookrightarrow]{r}{} \arrow[hookrightarrow]{rd}{} 
  & \overline{Z^{(0)}_2} \arrow[hookrightarrow]{ru}{} \arrow[hookrightarrow]{rd}{} & \\
  & \overline{Z^{(2)}_2} \arrow[hookrightarrow]{r}{} & \overline{Z^{(2)}_3}
  \end{tikzcd}
  \right\}.
\]
\caption{Inclusion relationships between closures of the non-stable $PGL_2$-orbits of $\Gr(2,4)$}
\label{fig-PGL2-non-stable}
\end{figure}

Unfortunately, the stratification method in this case is greatly complicated by the overlapping closed strata, since we would have to deal with finding relations between generators that are pushforwards from different closed strata. For example, to compute $A^*_{\PGL_2}( \overline{Z^{(1)}_3} )$, we need the following exact sequences:
\begin{align*}
A^*_{\PGL_2}(Z_1) &\xrightarrow{} A^*_{\PGL_2}(\overline{Z^{(1)}_2})
\xrightarrow{} A^*_{\PGL_2}(Z^{(1)}_2)
\rightarrow 0, \\
A^*_{\PGL_2}(Z_1) &\xrightarrow{} A^*_{\PGL_2}(\overline{Z^{(0)}_2})
\xrightarrow{} A^*_{\PGL_2}(Z^{(0)}_2)
\rightarrow 0, \\
A^*_{\PGL_2}(Z_1) 
&\rightarrow A^*_{\PGL_2}(\overline{Z^{(1)}_2}) \oplus A^*_{\PGL_2}(\overline{Z^{(0)}_2})
\xrightarrow{} A^*_{\PGL_2}(\overline{Z^{(1)}_2}\cup\overline{Z^{(0)}_2})
\rightarrow 0, \\
A^*_{\PGL_2}(\overline{Z^{(1)}_2}\cup\overline{Z^{(0)}_2}) &\xrightarrow{} A^*_{\PGL_2}(\overline{Z^{(1)}_3})
\xrightarrow{} A^*_{\PGL_2}(Z^{(1)}_3)
\rightarrow 0.
\end{align*}
As such, we will need to find relations between pushforwards of generators of $A^*_{\PGL_2}(\overline{Z^{(1)}_2})$ and $A^*_{\PGL_2}(\overline{Z^{(0)}_2})$, which is a difficult problem since they both involve generators of equivariant Chow rings of two different orbits.

There is one other difficulty with finding relations once all the generators are found by the stratification method, and we illustrate this by the following example with $A^1_{\PGL_2}(\Gr(2,4))$.
\begin{myexample}
\label{example-Picard-gp-Grassmannian-of-lines}

From Proposition \ref{closure-properties-of-Zi-in-Grassmannian}, the excision exact sequence for computing $A^1_{\PGL_2}(\Gr(2,4))$ is given by
\[
A^0_{\PGL_2}(\overline{Z^{(1)}_3}\cup\overline{Z^{(2)}_3}) \xrightarrow{} A^1_{\PGL_2}(\Gr(2,4))
\xrightarrow{} A^1_{\PGL_2}(\Gr(2,4)^s)
\rightarrow 0.
\]
Roughly speaking, the image of the leftmost pushforward map is generated by the $\PGL_2$-equivariant fundamental classes $[\overline{Z^{(1)}_3}]^{\PGL_2}$ and $[\overline{Z^{(2)}_3}]^{\PGL_2}$ in $A^1_{\PGL_2}(\Gr(2,4))$. The proof of Proposition \ref{closure-properties-of-Zi-in-Grassmannian} gives
\begin{align*}
Z^{(1)}_3 \cup Z^{(1)}_2 &= \phi^{-1}(6^3:1),\\
Z^{(2)}_3 \cup Z^{(2)}_2 &= \phi^{-1}(-6^3:1),
\end{align*}
where 
\[
\phi: \Gr(2,4)^{ss} \rightarrow \mathbb{P}^1, p \mapsto (I'^3|_p:J|_p)
\]
is the good GIT quotient of $\Gr(2,4)$ from \cite{Ne81}. It follows that
\begin{align*}
\overline{Z^{(1)}_3} &= V(I'^3-6^3J) \subset \Gr(2,4), \\
\overline{Z^{(2)}_3} &= V(I'^3+6^3J) \subset \Gr(2,4).
\end{align*}
Similarly, the closure of every $\PGL_2$-stable orbit $Z\subset \Gr(2,4)^s$ such that $\phi(Z) = (x:y)$ is given by
\[
\overline{Z} = V(yI'^3-xJ) \subset \Gr(2,4).
\]
Consider the Wronskian map
\[
\Wr: \Gr(2,4)^{ss} \rightarrow \mathbb{P}^4
\]
in the form given by Newstead \cite{Ne81} as introduced in Section \ref{subsection-basics-GIT-for-pencils-of-binary-cubics}. The above shows that $\overline{Z^{(1)}_3}$, $\overline{Z^{(2)}_3}$ as well as every $\overline{Z} \neq \overline{\mathcal{O}_{A_4}}$ is a section of the pullback of the line bundle $\mathcal{O}_{\mathbb{P}^4}(3)$ over $\mathbb{P}^4$ via $\Wr$. 

Since $\mathcal{O}_{\mathbb{P}^4}(3)$ is $\GL_2$-equivariant, we have that
\[
[\overline{Z^{(1)}_3}]^{\GL_2} = [\overline{Z^{(2)}_3}]^{\GL_2} = [\overline{Z}]^{\GL_2}
\]
whenever $Z\neq \mathcal{O}_{A_4}$. However, it is not evident whether this relation still holds in the $\PGL_2$-equivariant setting, given that $\mathcal{O}_{\mathbb{P}^4}(3)$ is not $\PGL_2$-equivariant. In general, we would end up with an excessive number of generators via the stratification method given such difficulties with determining whether relations in $A^*_{\GL_2}(Gr(2,4))$ pull back to $A^*_{\PGL_2}(\Gr(2,4))$. 

It is possible that techniques from \cite{ST22} for computing $A^*_{\PGL_2}(\mathbb{P}^n)$ would help with our computations. However, it does not appear that results from \cite{ST22} give immediate answers to the problem of computing $A^*_{\PGL_2}(\Gr(2,4))$, since the Wronskian does not induce a surjective pullback homomorphism of Chow rings. For example, we have
\[
\Wr(Z^{(1)}_3\cup Z^{(1)}_2) = \Wr(Z^{(2)}_3\cup Z^{(2)}_2)
\]
from the above, so we would get
\[
[Z^{(1)}_3\cup Z^{(1)}_2]^{\PGL_2} + [Z^{(2)}_3\cup Z^{(2)}_2]^{\PGL_2} = {\Wr}^*[\Wr(Z^{(1)}_3\cup Z^{(1)}_2)]^{\PGL_2}
\]
where the class $[\Wr(Z^{(1)}_3\cup Z^{(1)}_2)]^{\PGL_2} \in A^*_{\PGL_2}(\mathbb{P}^4)$ is computable. Hence, it still remains to determine relations between $[\overline{Z^{(1)}_3}]^{\PGL_2}$ and $[\overline{Z^{(2)}_3}]^{\PGL_2}$ in $A^*_{\PGL_2}(\Gr(2,4))$.
\end{myexample}

\end{document}